\documentclass[smallextended]{svjour3}

\usepackage{amsmath}
\usepackage{mathrsfs}
\usepackage{graphicx}
\usepackage{amssymb}
\usepackage{dsfont}
\usepackage{a4}
\usepackage{url}
\usepackage{pstricks}
\usepackage{color}

\newcommand{\bea}{\begin{eqnarray*}}
\newcommand{\eea}{\end{eqnarray*}}
\newcommand{\be}{\begin{eqnarray}}
\newcommand{\ee}{\end{eqnarray}}

\def\dd{{\rm {d}}}

\def\var{\mathsf{var}}

\def\mp{\partial}
\def\tr{\mbox{\rm trace}}
\def\Ex{\mathsf{E}}
\def\Fx{\mathsf{F}}

\def\rank{\mbox{\rm rank}}

\def\Mb{{\mathbf M}}

\def\1b{{\mathbf 1}}
\def\ra{\rightarrow}

\def\TT{^\top}

\def\ml{\lambda}
\def\ma{\alpha}

\def\0b{{\mathbf 0}}

\def\Xb{{\mathbf X}}
\def\Yb{{\mathbf Y}}
\def\Zb{{\mathbf Z}}

\def\SM{{\mathscr M}}
\def\SMM{\mathbb{M}}
\def\SN{{\mathscr N}}

\def\SV{{\mathscr V}}

\newcommand{\fin} {\mbox{}~\hfill\mbox{$\Box$}}
\newcommand{\carre} {\mbox{}~\hfill\rule{2mm}{2mm}}

\newtheorem{thm}{Theorem}

\newcommand{\vsp}{\vspace{0.3cm}}

\begin{document}

\title{Bregman divergences based on optimal design criteria and simplicial measures of dispersion}
\titlerunning{Bregman divergences and simplicial measures of dispersion}

\author{Luc Pronzato
\and Henry P. Wynn
\and Anatoly~Zhigljavsky}

\institute{Luc Pronzato \at
              CNRS, UCA, Laboratoire I3S, UMR 7172; 2000, route des Lucioles, Les Algorithmes, b\^at. Euclide B, 06900 Sophia Antipolis, France \\
              Tel.: +33-4-89154345\\
              \email{Luc.Pronzato@cnrs.fr}           
           \and
           H.P. Wynn \at
              London School of Economics, Houghton Street, London, WC2A 2AE, UK \\
              \email{H.Wynn@lse.ac.uk}
            \and
            A. Zhigljavsky \at
            School of Mathematics, Cardiff University, Senghennydd Road, Cardiff, CF24 4YH, UK\\
            \email{ZhigljavskyAA@cf.ac.uk}
}

\date{Received: date / Accepted: date}

\maketitle 

\begin{abstract}
In previous work the authors defined the $k$-th order simplicial distance between probability distributions which arises naturally from a measure of dispersion based
on the squared volume of random simplices of dimension $k$. This theory is embedded in the wider theory of divergences and distances between distributions which includes Kullback-Leibler, Jensen-Shannon, Jeffreys-Bregman divergence and Bhattacharyya distance. A general construction is given based on defining
a directional derivative of a function $\phi$ from one distribution to the other whose concavity or strict concavity influences the properties of the resulting divergence. For the normal distribution these divergences can be expressed as matrix formula for the (multivariate) means and covariances. Optimal experimental design criteria contribute a range of functionals applied to non-negative, or positive definite, information matrices. Not all can distinguish normal distributions but sufficient conditions are given. The $k$-th order simplicial distance is revisited from this aspect and the results are used to test empirically the identity of means and covariances.


\end{abstract}

\section{Introduction}\label{S:intro}

There are close connections between divergences and distances between probability distributions, and certain Fr\'echet-type derivatives. Moreover, for the normal distributions and for the information matrices which dominate the theory of optimal experimental design, the distances can be expressed in matrix form. A natural question that the paper explores is which distance, or which type of experimental design criterion, is best able to distinguish between two normal distributions, in particular when the covariance matrices are close to singularity.

\vsp
Many divergences and distances between probability distributions are constructed from concave functionals $\phi$ defined on the set of probability measures, with the symmetrized Kullback-Leibler divergence, the Jensen-Shannon divergence and Bhattacharyya distance as typical examples; see, e.g., \cite{Basseville2013}, \cite{NielsenB2011}. Note that we shall call them distances also in the case when they only define semi-metrics; that is, when they do not satisfy the triangular inequality. Distances between two normal distributions only depend on their first two moments. One can thus derive simple statistical criteria for testing the identity of means and covariances matrices of two distributions based on two samples, using empirical estimates.

In the same way, design of optimal experiments relies on the maximization of concave functions $\Phi$ of information matrices, see the abundant literature on the subject \cite{AtkinsonDT2007}, \cite{Fedorov72}, \cite{FedorovH97}, \cite{FedorovL2014}, \cite{Pazman86}, \cite{PP2013}, \cite{Pukelsheim93}, \cite{Silvey80}. Such concave, sometimes strictly concave, design criteria form natural candidates for the definition of distances between two normal distributions. Also, in a recent paper \cite{PWZ-2018-jmva} we considered simplicial distances induced by the dispersion functionals
\bea 
\phi_k(\mu)= \Ex_\mu\{ \SV_k^2(X_0,\ldots,X_{k}) \}\,,
\eea
where $\SV_k(X_0,X_1,\ldots,X_{k})$ is the volume of the $k$-dimensional simplex (its length when $k=1$ and area when $k=2$) formed by the $k+1$ vertices $X_0,X_1,\ldots,X_{k}$ assumed to be i.i.d.\ with $\mu$ in $\mathds{R}^d$. The functional $\phi_k^{1/k}$ is concave \cite{PWZ2017}, and may thus also be considered for the construction of distances between distributions.

The paper explores the connections between the various notions of distances induced by these approaches. In particular, we show that the construction of $\phi_k$, based on  volumes of $k$-dimensional simplices, makes the associated distances more sensitive to the dimensionality of the data than other, more usual, distances between normal distributions, Bhattacharyya distance for instance. We also show that Kiefer's family of design criteria $\varphi_p$ with $p>0$, which are rather insensitive to the presence of small eigenvalues, may conveniently be used to define distances between normal distributions, in particular for measures concentrated in small dimension subspaces.


\section{Distances defined from concave functionals}\label{S:concave functionals}

Let $\phi$ denote a twice-continuously Fr\'echet-differentiable real-valued concave functional defined on the set $\SM$ of probability measures on Borel sets of $\mathds{R}^d$. For any $\mu,\zeta\in\SM$, denote by $F_\phi(\mu,\zeta)$ the directional derivative of $\phi$ at $\mu$ in the direction $\zeta$,
\be\label{dir-der}
F_\phi(\mu,\zeta) = \lim_{\ma\ra 0^+} \frac{\phi[(1-\ma)\mu+\ma\xi]-\phi(\mu)}{\ma} \,;
\ee
that is, the Fr\'echet derivative of $\phi$ at $\mu$ in the direction $\xi-\mu$, see for instance \cite{FrigyikSG2008b}.
The Bregman divergence between $\mu$ and $\zeta$ associated with $\phi$ is then
\bea 
D_{\phi,B}(\mu,\zeta)=\phi(\mu)+F_\phi(\mu,\zeta)-\phi(\zeta) \,,
\eea
and the strict concavity of $\phi$ implies that $D_{\phi,B}(\mu,\xi) \geq 0$ with $D_{\phi,B}(\mu,\xi)=0$ if and only if $\zeta=\mu$. When $\phi$ is strictly concave on $\SM$, the Jeffreys-Bregman divergence
\be\label{JB}
D_{\phi,JB}(\mu,\zeta)=\frac12 \left[D_{\phi,B}(\mu,\zeta)+D_{\phi,B}(\zeta,\mu)\right]= \frac12 \left[F_\phi(\mu,\zeta)+F_\phi(\zeta,\mu)\right] \,,
\ee
obtained by symmetrization, and the Burbea-Rao divergence
\be\label{BR}
D_{\phi,BR}(\mu,\zeta)= \phi\left(\frac{\mu+\zeta}{2}\right) - \frac{\phi(\mu)+\phi(\zeta)}{2} \,,
\ee
which does not require $\phi$ to be Fr\'echet-differentiable, define semi-metrics on $\SM$; see for instance \cite{Basseville2013}, \cite{NielsenB2011}. A classical example in the case when $\mu$ and $\zeta$ have densities $\varpi_\mu$ and $\varpi_\zeta$ on $\mathds{R}^d$ with respect to the Lebesgue measure, is given by $\phi$ equal to the Shannon entropy $H_0$, with
$$
H_0(\mu)=-\int \log[\varpi_\mu(x)] \,\varpi_\mu(x)\, \dd x\,.
$$
The Jeffreys-Bregman divergence $D_{H_0,JB}$ is then simply the symmetrized Kullback-Leibler divergence
$$
D_{KL}(\mu,\zeta) = \frac12\, [KL(\mu\|\zeta)+KL(\zeta\|\mu)] \,,
$$
and $D_{H_0,BR}(\mu,\zeta)$ coincides with the Jensen-Shannon divergence,
$$
D_{JS}(\mu,\zeta)=\frac12 \, \left\{KL\left[\mu\|\left(\frac{\mu+\zeta}{2}\right)\right]+KL\left[\zeta\|\left(\frac{\mu+\zeta}{2}\right)\right]\right\} \,,
$$
where $KL(\mu\|\zeta)=\int \log[\varpi_\mu(x)/\varpi_\zeta(x)] \,\varpi_\mu(x)\, \dd x$. More generally, one can define $KL(\mu\|\zeta)=\int \log[\dd\mu/\dd\zeta](x) \,\dd\mu(x)$ if $\mu\ll \zeta$ (i.e., if $\zeta$ dominates $\mu$) and $KL(\mu\|\zeta)=+\infty$ otherwise; see \cite[Sect.~III.9]{Shiryaev96}.
Also, the Hellinger integral $H(\mu,\zeta)$ is defined by
$$
H(\mu,\zeta) = \int \sqrt{\frac{\dd\mu}{\dd\nu}(x)\frac{\dd\zeta}{\dd\nu}(x)}\, \dd\nu(x) \,,
$$
with $\nu$ denoting any dominating measure for $\mu$ and $\zeta$, and $\rho(\mu,\zeta)=\sqrt{1-H(\mu,\zeta)}$ defining a metric on $\SM$. The Bhattacharyya distance $D_B(\mu,\zeta)=-\log H(\mu,\zeta)$ defines a semi-metric on $\SM$; see \cite{Bhattacharyya46}, \cite[Sect.~III.9]{Shiryaev96}.

When $\mu$ and $\zeta$ correspond to normal distributions, the distances defined above only depend on their respective means $a_\mu$ and $a_\zeta$ and covariances $\Sigma_\mu$ and $\Sigma_\zeta$ (we assume that $\Sigma_\mu$ and $\Sigma_\zeta$ have full rank $d$). In particular, $D_{KL}$, $D_{JS}$ and $D_B$ take simple expressions:
\be
D_{KL}(\mu,\zeta) &=& \frac14\, \left[ \tr(\Sigma_\mu^{-1}\Sigma_\zeta) + \tr(\Sigma_\zeta^{-1}\Sigma_\mu) \right] \nonumber \\
&& + \frac14\, (a_\zeta-a_\mu)\TT(\Sigma_\mu^{-1}+\Sigma_\zeta^{-1})(a_\zeta-a_\mu) - \frac{d}{2}\,,  \label{DKL} \\
D_{JS}(\mu,\zeta) &=& \frac12 \, \log\left[
\frac{\det\left(\frac{\Sigma_\mu+\Sigma_\zeta}{2}\right)}{\sqrt{\det(\Sigma_\mu)\det(\Sigma_\zeta)}}
 \right]  \nonumber \\
 &&  + \frac12 \, \log\left[ 1+\frac12 (a_\mu-a_\zeta)\TT(\Sigma_\mu+\Sigma_\zeta)^{-1}(a_\mu-a_\zeta)\right]\,, \label{DJS} \\
D_B(\mu,\zeta) &=& \frac12\, \log\left[
\frac{\det\left(\frac{\Sigma_\mu+\Sigma_\zeta}{2}\right)}{\sqrt{\det(\Sigma_\mu)\det(\Sigma_\zeta)}}
 \right] \nonumber \\
 && +\, \frac14 (a_\mu-a_\zeta)\TT(\Sigma_\mu+\Sigma_\zeta)^{-1}(a_\mu-a_\zeta) \,. \label{DB}
\ee
Note that $D_{B}(\mu,\zeta) \geq D_{JS}(\mu,\zeta)$, with equality when $a_\zeta=a_\mu$, and that $D_{KL}$, $D_{JS}$ and $D_B$ satisfy the following invariance property
\be\label{invariance}
D(\mu,\zeta)=D(\mu_0,\zeta[\mu]) \,,
\ee
where $\mu_0$ has zero mean and covariance $I_d$, the $d$-dimensional identity matrix, and $\zeta[\mu]$ has mean $\Sigma_\mu^{-1/2}(a_\zeta-a_\mu)$ and covariance $\Sigma_\mu^{-1/2}\Sigma_\zeta\Sigma_\mu^{-1/2}$. One may refer to \cite{NielsenN2017} for more detailed developments on connections between Bhattacharyya distance and other divergence measures.

For each of these distances, $D_{KL}(\mu,\zeta)$, $D_{JS}(\mu,\zeta)$ and $D_B(\mu,\zeta)$, equality to zero is obtained if and only if $a_\zeta=a_\mu$ and $\Sigma_\zeta=\Sigma_\mu$. When a distance $D$ satisfies this property, we shall say that $D$ distinguishes normal distributions.

\section{Distances based on optimal design criteria}\label{S:general-OD-criteria}

Optimal design of experiments rely on the maximization of a concave functional $\Phi$ of the information matrix. Below we show that some classical optimality criteria, such as A- and D-optimality, yield distance measures that are able to distinguish normal distributions, but that the usual notion of strict concavity used in optimal design theory is not enough to obtain this property.

\subsection{Construction}\label{S:construction}
Denote by $\SMM^>$ (respectively, $\SMM^\geq$) the set of $d\times d$ symmetric positive definite (respectively, non-negative definite) matrices. In Section~\ref{S:k-simplicial} we shall consider functions $\Phi$ whose properties depend on the rank of the matrices involved; $\SMM^*$ will denote a general matrix cone included in $\SMM^\geq$ such that $\Mb_1+\Mb_2\in \SMM^*$ for any $\Mb_1\in\SMM^*$ and $\Mb_2\in\SMM^\geq$. We shall consider two particular cases: $\SMM^*=\SMM^>$, and $\SMM^*=\SMM^r$, the subset of $\SMM^\geq$ containing matrices of rank at least $r \leq d$. We denote by $\SM^*$ the subset of $\SM$ containing distributions with finite covariances in $\SMM^*$, with $\SM^>$ and $\SM^r$ as particular cases associated with $\SMM^*=\SMM^>$ and $\SMM^*=\SMM^r$.

Let $\Phi$ be a function defined on $\SMM^\geq$, isotonic on $\SMM^*$ relative to the Loewner ordering ($\Phi(M_1)\geq\Phi(M_2)$ when $M_2\in\SMM^*$ and $M_1-M_2\in\SMM^\geq$) and concave on $\SMM^*$ ($\Phi[(1-\ma)M_1+\ma M_2] \geq (1-\ma)\Phi(M_1) + \ma\Phi(M_2)$ for all $\ma\in(0,1)$ and $M_1\in\SMM^*$, $M_2\in\SMM^\geq$).
Consider two probability measures $\mu$ and $\zeta$ with respective means $a_\mu$ and $a_\zeta$ and covariances $\Sigma_\mu=\var(\mu)$ and $\Sigma_\zeta=\var(\zeta)$, with $\Sigma_\mu, \Sigma_\zeta\in\SMM^*$. Following \eqref{BR}, the Burbea-Rao divergence $D_{\Phi,BR}(\mu,\zeta)$ associated with $\Phi$ is defined by
\bea 
D_{\Phi,BR}(\mu,\zeta)= \Phi\left[\var\left(\frac{\mu+\zeta}{2}\right)\right] - \frac{\Phi(\Sigma_\mu)+\Phi(\Sigma_\zeta)}{2} \,.
\eea
Direct calculation gives
\bea
\var[(\mu+\zeta)/2]=\frac12\, (\Sigma_\mu+\Sigma_\zeta)+ \frac14\, (a_\zeta-a_\mu)(a_\zeta-a_\mu)\TT\,,
\eea
so that
\be\label{BR2}
\hspace{-0.7cm} D_{\Phi,BR}(\mu,\zeta)=\Phi\left[\frac12\, (\Sigma_\mu+\Sigma_\zeta)+ \frac14\,(a_\zeta-a_\mu)(a_\zeta-a_\mu)\TT \right] - \frac{\Phi(\Sigma_\mu)+\Phi(\Sigma_\zeta)}{2} \,,
\ee
with $D_{\Phi,BR}(\mu,\zeta)\geq~0$ from the isotonicity and concavity of $\Phi$ on $\SMM^*$.

Denote now
\bea
F_\Phi(\mu,\zeta) = \lim_{\ma\ra 0^+} \frac{\Phi\{\var[(1-\ma)\mu+\ma\xi]\}-\Phi[\var(\mu)]}{\ma} \,,
\eea
see \eqref{dir-der}. For $\ma\in[0,1]$, define $\mu_{x,\ma}=(1-\ma)\mu+\ma\delta_x$, with $\delta_x$ the Dirac delta measure at $x$. Straightforward calculation gives
\bea
\frac{\mp \var[\mu_{x,\ma}]}{\mp\ma}\bigg|_{\ma=0} = (x-a_\mu)(a-a_\mu)\TT - \Sigma_\mu \,,
\eea
so that, when $\Phi$ is differentiable at $\Sigma_\mu$, with gradient $\nabla_\Phi(\Sigma_\mu)$,
\be
F_\Phi(\mu,\zeta) &=& \int F_\Phi(\mu,\delta_x) \, \dd\zeta(x) =
\int \tr\left[\nabla_\Phi(\Sigma_\mu)\frac{\mp \var[\mu_{x,\ma}]}{\mp\ma}\bigg|_{\ma=0}\right]\, \dd\zeta(x) \nonumber \\
&=& \tr[\nabla_\Phi(\Sigma_\mu)(\Sigma_\zeta-\Sigma_\mu)] + (a_\zeta-a_\mu)\TT\nabla_\Phi(\Sigma_\mu)(a_\zeta-a_\mu) \,. \label{Dirdergeneral}
\ee
Similarly to \eqref{JB}, the Jeffreys-Bregman divergence $D_{\Phi,JB}(\mu,\zeta)$ associated with $\Phi$ is then defined as
\be
D_{\Phi,JB}(\mu,\zeta) &=& \frac12\, [F_\Phi(\mu,\zeta)+F_\Phi(\zeta,\mu)] \nonumber \\
&=& \frac12\, \left[ \tr\{[\nabla_\Phi(\Sigma_\mu)-\nabla_\Phi(\Sigma_\zeta)](\Sigma_\zeta-\Sigma_\mu)\} \right. \nonumber \\
&& \left. +\, (a_\zeta-a_\mu)\TT[\nabla_\Phi(\Sigma_\mu)+\nabla_\Phi(\Sigma_\zeta)](a_\zeta-a_\mu) \right] \,. \label{JB2}
\ee
For any $z\in\mathds{R}^d$ with $\|z\|=1$, we have $\Phi(\Sigma_\mu+zz\TT)\leq \Phi(\Sigma_\mu)+ z\TT \nabla_\Phi(\Sigma_\mu) z$ from the concavity of $\Phi$ on $\SMM^*$ and $\Phi(\Sigma_\mu+zz\TT)\geq \Phi(\Sigma_\mu)$ from its isotonicity. Therefore, $z\TT \nabla_\Phi(\Sigma_\mu) z\geq 0$, and $\nabla_\Phi(\Sigma_\mu)\in\SMM^\geq$. Similarly, $\nabla_\Phi(\Sigma_\zeta)\in\SMM^\geq$, showing that the second term in \eqref{JB2} is non-negative. Concavity on $\SMM^*$ also implies
$\Phi(\Sigma_\zeta)\leq \Phi(\Sigma_\mu)+\tr[\nabla_\Phi(\Sigma_\mu)(\Sigma_\zeta-\Sigma_\mu)]$ and $\Phi(\Sigma_\mu)\leq \Phi(\Sigma_\zeta)+\tr[\nabla_\Phi(\Sigma_\zeta)(\Sigma_\mu-\Sigma_\zeta)]$, which gives $\tr\{[\nabla_\Phi(\Sigma_\mu)-\nabla_\Phi(\Sigma_\zeta)](\Sigma_\zeta-\Sigma_\mu)\}\geq 0$. Therefore, $D_{\Phi,JB}(\mu,\zeta)\geq 0$.

Below we investigate which additional conditions must be imposed on $\Phi$ to ensure that $D_{\Phi,BR}$ \eqref{BR2} and $D_{\Phi,JB}$ \eqref{JB2} distinguish normal distributions in $\SM^*$.

\subsection{Sufficient conditions for distinguishability}\label{S:distinguishability}
We say that $\Phi$ is positively homogeneous when
\bea 
\Phi(\ma M)=\ma\, \Phi(M) \mbox{ for any } \ma>0 \mbox{ and } M\in\SMM^\geq \,,
\eea
and we shall say that $\Phi$ is strictly isotonic on $\SMM^*$ when
\bea
\Phi(M_1)>\Phi(M_2) \mbox{ for any } M_1,M_2 \mbox{ such that } M_1-M_2\in\SMM^\geq,\, M_2\in\SMM^*,\, M_2 \neq M_1 \,.
\eea
In optimal design of experiments, a function $\Phi$ is said to be strictly concave on the cone $\SMM^*\subset\SMM^\geq$ when
\be
\Phi[(1-\ma)M_1+\ma M_2] &>& (1-\ma)\Phi(M_1) + \ma\,\Phi(M_2) \nonumber \\
&& \mbox{ for all } \ma\in(0,1)\,, \ M_1\in\SMM^* \mbox{ and } M_2\in\SMM^\geq \label{Concavity1}\\
&& \mbox{ with } M_2\neq 0 \mbox{ and } M_2 \mbox{ not proportional to } M_1 \,. \nonumber
\ee
When $\SMM^*=\SMM^>$, this definition coincides with that in \cite[Sect.~5.2]{Pukelsheim93}.
%
The usual definition in convex analysis is stronger and requires the inequality to be valid for a wider class of matrices $M_2$. We shall call \emph{strongly strictly} concave on $\SMM^*$ a function $\Phi$ such that
\be
\Phi[(1-\ma)M_1+\ma M_2] &>& (1-\ma)\Phi(M_1) + \ma\,\Phi(M_2) \nonumber \\
&& \mbox{ for all } \ma\in(0,1)\,, \ M_1\in\SMM^* \mbox{ and } M_2\in\SMM^\geq \label{Concavity2}\\
&& \mbox{ with } M_2\neq 0 \mbox{ and } M_2 \neq M_1 \,. \nonumber
\ee
The following property shows that $D_{\Phi,BR}$ \eqref{BR2} and $D_{\Phi,JB}$ \eqref{JB2} distinguish normal distributions in $\SM^*$ when $\Phi$ is strictly isotonic and strongly strictly concave on $\SMM^*$.

\begin{lemma}\label{L1}
Let $\Phi$ be a strictly isotonic and strongly strictly concave function on $\SMM^*$. Then, for $\mu$ and $\zeta$ two probability measures with respective means $a_\mu$ and $a_\zeta$ and covariances $\Sigma_\mu=\var(\mu)$ and $\Sigma_\zeta=\var(\zeta)$, $\Sigma_\mu, \Sigma_\zeta \in\SMM^*$, we have
\be
D_{\Phi,BR}(\mu,\zeta)=0 \Rightarrow a_\mu=a_\zeta \mbox{ and } \Sigma_\mu=\Sigma_\zeta \,, \label{BR-disting}
\ee
with $D_{\Phi,BR}$ given by \eqref{BR2}, and, when $\Phi$ is differentiable at $\Sigma_\mu$ and $\Sigma_\zeta$,
\be
D_{\Phi,JB}(\mu,\zeta)=0 \Rightarrow a_\mu=a_\zeta \mbox{ and } \Sigma_\mu=\Sigma_\zeta \,, \label{JB-disting}
\ee
where $D_{\Phi,JB}$ is defined by \eqref{JB2}.
\end{lemma}

\begin{proof}
We first prove that $D_{\Phi,BR}(\mu,\zeta)=0$, or $D_{\Phi,JB}(\mu,\zeta)=0$, implies $a_\mu=a_\zeta$.
Suppose that $a_\mu \neq a_\zeta$. The strict isotonicity of $\Phi$ on $\SMM^*$ implies $\Phi[(\Sigma_\mu+\Sigma_\zeta)/2+ (a_\zeta-a_\mu)(a_\zeta-a_\mu)\TT/4 ]>\Phi[(\Sigma_\mu+\Sigma_\zeta)/2]$, and therefore $D_{\Phi,BR}(\mu,\zeta)>0$ from concavity.
Take any $z\in\mathds{R}^d$ with $\|z\|=1$. We have $\Phi(\Sigma_\mu+zz\TT) \leq \Phi(\Sigma_\mu)+ z\TT \nabla_\Phi(\Sigma_\mu) z$ from the concavity of $\Phi$ and $\Phi(\Sigma_\mu+zz\TT)> \Phi(\Sigma_\mu)$ from its strict isotonicity. Therefore, $z\TT \nabla_\Phi(\Sigma_\mu) z> 0$, and $\nabla_\Phi(\Sigma_\mu)\in\SMM^>$. Similarly, $\nabla_\Phi(\Sigma_\zeta)\in\SMM^>$, showing that $(a_\zeta-a_\mu)\TT[\nabla_\Phi(\Sigma_\mu)+\nabla_\Phi(\Sigma_\zeta)](a_\zeta-a_\mu)$ in \eqref{JB2} is strictly positive.

We consider now two distributions such that $a_\mu=a_\zeta$ and $\Sigma_\mu \neq \Sigma_\zeta$.
Since $\Phi$ is strongly strictly concave on $\SMM^*$, $\Phi[(\Sigma_\mu+\Sigma_\zeta)/2] > [\Phi(\Sigma_\mu)+\Phi(\Sigma_\zeta)]/2$ for $\Sigma_\zeta \neq \Sigma_\mu$, which concludes the proof of \eqref{BR-disting}. Also, for $\Sigma_\zeta \neq \Sigma_\mu$ we have $\Phi(\Sigma_\zeta)<\Phi(\Sigma_\mu)+\tr[\nabla_\Phi(\Sigma_\mu)(\Sigma_\zeta-\Sigma_\mu)]$ and $\Phi(\Sigma_\mu)<\Phi(\Sigma_\zeta)+\tr[\nabla_\Phi(\Sigma_\zeta)(\Sigma_\mu-\Sigma_\zeta)]$, so that $\tr\{[\nabla_\Phi(\Sigma_\mu)-\nabla_\Phi(\Sigma_\zeta)](\Sigma_\zeta-\Sigma_\mu)\}> 0$ and $D_{\Phi,JB}(\mu,\zeta)> 0$, which proves~\eqref{JB-disting}.
\carre
\end{proof}

A positively homogeneous function $\Phi$ is not strongly strictly concave. Indeed, take $M_2=\beta M_1$, with $M_1\in\SMM^*$, $\beta>0$ and $\beta\neq 1$.
We have $\Phi[(1-\ma)M_1+\ma M_2]=\Phi[(1-\ma+\ma\beta)M_1]=(1-\ma+\ma\beta)\Phi(M_1)=(1-\ma)\Phi(M_1) + \ma\,\Phi(M_2)$. An important consequence is that the Burbea-Rao and Jeffreys-Bregman divergences associated with a strictly concave (in the sense of \eqref{Concavity1}) and positively homogeneous function $\Phi$ are unable to distinguish normal distributions. Take $\mu$ and $\zeta$ such that $a_\mu=a_\zeta$ and $\Sigma_\zeta = \beta\, \Sigma_\mu$, $\beta>0$ and $\beta\neq 1$. One can readily check that $D_{\Phi,BR}=0$, see \eqref{BR2}. Also, when $\Phi$ is differentiable at $\Sigma_\mu$, then $\nabla_\Phi(\Sigma_\mu)=\nabla_\Phi(\Sigma_\zeta)$ and $D_{\Phi,JB}=0$, see \eqref{JB2}. In contrast, the following property shows that $D_{\Phi,BR}$ and $D_{\Phi,JB}$ do distinguish normal distributions when using $\log\Phi$ instead of $\Phi$.

\begin{lemma}\label{L2}
Let $\Phi$ be a function positively homogeneous, non identically zero, strictly isotonic on $\SMM^>$, and strictly concave in the sense of \eqref{Concavity1}. Then,
for $\mu$ and $\zeta$ two probability measures with respective means $a_\mu$ and $a_\zeta$ and covariances $\Sigma_\mu=\var(\mu)$ and $\Sigma_\zeta=\var(\zeta)$, $\Sigma_\mu, \Sigma_\zeta \in\SMM^>$, we have
\bea
D_{\log\Phi,BR}(\mu,\zeta)=0 \Rightarrow a_\mu=a_\zeta \mbox{ and } \Sigma_\mu=\Sigma_\zeta \,, 
\eea
and when $\Phi$ is differentiable at $\Sigma_\mu$ and $\Sigma_\zeta$,
\bea
D_{\log\Phi,JB}(\mu,\zeta)=0 \Rightarrow a_\mu=a_\zeta \mbox{ and } \Sigma_\mu=\Sigma_\zeta \,. 
\eea
\end{lemma}

\begin{proof} First note that $\Phi(\Sigma_\mu)>0$ and $\Phi(\Sigma_\zeta)>0$ since $\Sigma_\mu, \Sigma_\zeta \in\SMM^>$, see \cite[Chap.~5]{Pukelsheim93}, so that $\log\Phi(\Sigma_\mu)$ and $\log\Phi(\Sigma_\zeta)$ are well defined. Also, when $\Phi$ is differentiable at $\Sigma$, $\log\Phi$ is differentiable too, with
$\nabla_{\log\Phi}(\Sigma)=\nabla_\Phi(\Sigma)/\Phi(\Sigma)$.

Using Lemma~\ref{L1}, we only need to show that $\log\Phi$ is a strictly isotonic and strongly strictly concave function on $\SMM^>$. Strict isotonicity follows from the fact that the logarithm is increasing. Consider now \eqref{Concavity2}. Take any $M_1\in\SMM^>$ and $M_2\in\SMM^\geq$, $M_2\neq 0$, and any $\ma\in(0,1)$. We can write
\be
\log\Phi[(1-\ma)M_1+\ma M_2] &\geq& \log[(1-\ma)\Phi(M_1) + \ma\,\Phi(M_2)] \nonumber \\
&& \geq (1-\ma)\log\Phi(M_1) + \ma\,\log\Phi(M_2) \,, \label{T1}
\ee
where the first inequality follows from the concavity of $\Phi$ and the second from the concavity of logarithm. From the monotonicity of logarithm and the strict concavity of $\Phi$ in the sense of \eqref{Concavity1}, equality between the two extreme terms implies $M_2=\beta\,M_1$ for some $\beta>0$. Since $\Phi$ is positively homogeneous, $\log\Phi[(1-\ma)M_1+\ma M_2]=(1-\ma)\log\Phi(M_1) + \ma\,\log\Phi(M_2)$ then gives $f(\beta)=\log(1-\ma+\ma\beta)-\ma\log(\beta)=0$. Direct calculation gives $\dd f(\beta)/\dd\beta=\ma[1/(1-\ma+\ma\beta)-1/\beta]$, showing that, for any $\ma\in(0,1)$, $f(\beta)$ has a unique minimum at $\beta=\beta_*=1$, with $f(\beta_*)=0$. Equality in \eqref{T1} thus implies $M_2=\beta_*\, M_1=M_1$, which proves \eqref{Concavity2}.
\carre
\end{proof}

\subsection{Optimal-design criteria}\label{S:ODcriteria}

Consider Kiefer's \cite{Kiefer74} $\varphi_p$-class of functions, $p\in\mathds{R}\cup\{-\infty,+\infty\}$, which defines a family of design criteria widely used in optimal design. For any $M\in\SMM^\geq$, $\varphi_p(M)$ is defined by
\bea 
\varphi_p(M)=\left\{
\begin{array}{ll}
\ml_{\max}(M) & \mbox{for } p = \infty \,, \\
\left[\frac1d\, \tr(M^p)\right]^{1/p} & \mbox{for } p\neq 0 \mbox{ and } p\neq \pm\infty \,,\\
\det^{1/d}(M) & \mbox{for } p = 0 \,, \\
\ml_{\min}(M) & \mbox{for } p = -\infty \,,
\end{array}
\right.
\eea
with $\varphi_p(M)=0$ if $M$ is singular when $p\leq 0$. A-optimal design corresponds to $p=-1$, D-optimal design to $p=0$ and E-optimal design to $p=-\infty$; $\varphi_p(I_d)=1$ for all $p$.
All $\varphi_p$ are positively homogeneous;
for $p\in(-\infty,1)$, $\varphi_p$ is differentiable and strictly isotonic on $\SMM^>$, and strictly concave in the sense of \eqref{Concavity1}, see \cite[Sect.~6.13]{Pukelsheim93}.
Lemma~\ref{L2} applies, and the Burbea-Rao and Jeffreys-Bregman divergences associated with $\log\varphi_p$, $p\in(-\infty,1)$, distinguish normal distributions in $\SM^>$. However, as Section~\ref{S:Application} will illustrate, distances associated with negative $p$ are very sensitive to the presence of small eigenvalues in the spectrum of covariances matrices, and are therefore not recommended. In contrast, the presence of zero eigenvalues $\ml_i(M)$ has little influence when $p>0$ as $\varphi_p(M)=[(1/d) \sum_{i: \ml_i(M)>0} \ml_i^p(M)]^{1/p}$.

The $\varphi_p$ are information functions and therefore satisfy $\tr\left[\nabla_{\varphi_p}(M)M\right]=\varphi_p(M)$ for $M\in\SMM^>$, see \cite[p.~168]{Pukelsheim93}, and we obtain
\bea
D_{\log\varphi_p,JB}(\mu,\zeta) &=& \frac12\, \left\{ \tr\left[\frac{\nabla_{\varphi_p}(\Sigma_\mu)}{\varphi_p(\Sigma_\mu)}\Sigma_\zeta\right]+
\tr\left[\frac{\nabla_{\varphi_p}(\Sigma_\zeta)}{\varphi_p(\Sigma_\zeta)}\Sigma_\mu\right]  \right\} \nonumber \\
&& + \frac12\, \left\{ (a_\zeta-a_\mu)\TT \left[\frac{\nabla_{\varphi_p}(\Sigma_\mu)}{\varphi_p(\Sigma_\mu)}+\frac{\nabla_{\varphi_p}(\Sigma_\zeta)}{\varphi_p(\Sigma_\zeta)}\right] (a_\zeta-a_\mu) \right\} - 1\,. \eea
For $p\neq 0$, we get
\be
D_{\log\varphi_p,JB}(\mu,\zeta) &=& \frac12\, \left[ \frac{\tr(\Sigma_\mu^{p-1}\Sigma_\zeta)}{\tr(\Sigma_\mu^p)} +
\frac{\tr(\Sigma_\zeta^{p-1}\Sigma_\mu)}{\tr(\Sigma_\zeta^p)} \right] \nonumber \\
&& \hspace{-1cm}  + \frac12 \, (a_\zeta-a_\mu)\TT\left(\frac{\Sigma_\mu^{p-1}}{\tr(\Sigma_\mu^p)}+\frac{\Sigma_\zeta^{p-1}}{\tr(\Sigma_\zeta^p)}\right)(a_\zeta-a_\mu) - 1 \,. \label{JB-phip}
\ee
This expression is also valid when $p=0$ with the convention $\tr(M^0)=d$. In general, \eqref{BR2} does not yield a simple expression for $D_{\log\varphi_p,BR}(\mu,\zeta)$.
For $p=0$, $\varphi_0(\Sigma)=\det^{1/d}(\Sigma)$ is directly related to the Shannon entropy $H_0$ of a normal distribution with covariance $\Sigma$, and we have
\bea
D_{\log\varphi_0,JB}(\mu,\zeta)= \frac2d\, D_{KL}(\mu,\zeta) \ \mbox{ and } \  D_{\log\varphi_0,BR}(\mu,\zeta)= \frac2d\, D_{JS}(\mu,\zeta)\,,
\eea
with $D_{KL}$ and $D_{JS}$ respectively given by  \eqref{DKL} and \eqref{DJS}. In general, $D_{\log\varphi_p,JB}$ and $D_{\log\varphi_p,BR}$ with $p\neq 0$ do not satisfy the invariance property \eqref{invariance}.

\section{$k$-th order simplicial distances}\label{S:k-simplicial}

\subsection{Squared volumes of $k$-dimensional simplices}

In a recent paper \cite{PWZ-2018-jmva}, we considered simplicial distances induced by the dispersion functionals
\be\label{phi_k0}
\phi_k(\mu)= \Ex_\mu\{ \SV_k^2(X_0,\ldots,X_{k}) \}\,,
\ee
where $\SV_k(X_0,X_1,\ldots,X_{k})$ is  the volume of the $k$-dimensional simplex (its length when $k=1$ and area when $k=2$) formed by the $k+1$ vertices $X_0,X_1,\ldots,X_{k}$ assumed to be i.i.d.\ with $\mu$ in $\mathds{R}^d$. In particular, for $k=1$ we have
$$
\phi_1(\mu) = \int\int \|x_1-x_2\|^2\, \mu(\dd x_1)\mu(\dd x_2) = 2\, \tr[\Sigma_\mu] \,,
$$
twice the trace of the covariance matrix of $\mu$. As shown below, when $k=d$ we get $\phi_d(\mu)=(d+1)/d!\, \det(\Sigma_\mu)$, which is proportional to the generalised variance widely used in multivariate statistics.

For any $M\in\SMM^\geq$, define
\be\label{psi_k}
\Phi_k(M)=\frac{k+1}{k!} \, e_k[\Lambda(M)] \,
\ee
with $\Lambda(M)$ the set of eigenvalues of $M$ and $e_k$ the elementary symmetric function of degree $k$ (with $e_0=1$). The following theorem is proved in \cite{PWZ2017}.

\begin{thm}
 \label{Th:Bernoulli}
For any $k\in\{1,\ldots,d\}$ and $\mu\in\SM$, we have $\phi_k(\mu)= \Phi_k[\var(\mu)]$. Moreover, the functional $\phi_k^{1/k}$ is concave on $\SM$.
\end{thm}

The $\Phi_k^{1/k}$, $k\in\{1,\ldots,d\}$, form a family of criteria between $\varphi_1=\Phi_1/(2d)$ and $\varphi_0= [d!/(d+1)]^{1/d}\ \Phi_d^{1/d}$. On the one hand, similarly to $\varphi_p$ with positive $p$, $\Phi_k(M)$ with $k$ small enough is relatively insensitive to the presence of small eigenvalues in $\Lambda(M)$. On the other hand, $\Phi_k(M)>0$ if and only if $M\in\SMM^k$ (i.e., $\rank(M)\geq k$), which makes the $\Phi_k$ more sensitive to the true dimensionality of the data than the $\varphi_p$ for $p\in[0,1]$.

\vsp
The expressions of $\Phi_k(M)$ and its gradient $\nabla_{\Phi_k}(M)$ at $M\in\SMM^\geq$ are given by
\bea
\Phi_k(M) &=& \frac{k+1}{k\, k!}\, \sum_{i=0}^{k-1} (-1)^{i-1}\, e_{k-i}[\Lambda(M)]\, \tr(M^i) \,,\\
\nabla_{\Phi_k}(M) &=& \frac{k+1}{k!}\, \sum_{i=0}^{k-1} (-1)^i\, e_{k-i-1}[\Lambda(M)]\, M^i\,,
\eea
see \cite{Lopez-FR-D98-MODA}, \cite{RodriguezDLF2003}.
In \cite{PWZ2017}, we show that the directional derivative of $\phi_k$ at $\mu$ in the direction $\zeta$ is given by \eqref{Dirdergeneral}, with the additional property $\tr[\nabla_{\Phi_k}(M)M]=k\, \Phi_k(M)$, $M\in\SMM^\geq$, which gives
\be\label{F_Phik}
F_{\Phi_k}(\mu,\zeta) =  \tr[\nabla_{\Phi_k}(\Sigma_\mu)\Sigma_\zeta] + (a_\zeta-a_\mu)\TT\nabla_{\Phi_k}(\Sigma_\mu)(a_\zeta-a_\mu) -k\, \Phi_k(\Sigma_\mu)\,.
\ee
One can readily check that $\Phi_k^{1/k}$ is positively homogeneous, it is therefore not strongly strictly concave, see Section~\ref{S:general-OD-criteria}. However, $\Phi_k^{1/k}$ is strictly isotonic on $\SMM^k$ \cite[Lemma~3]{PWZ-2018-jmva} and strictly concave in the sense of \eqref{Concavity1}  for $k\geq 2$ \cite[Lemma~6]{PWZ-2018-jmva}. Arguments similar to those in the proof of Lemma~\ref{L2} indicate that $\log\Phi_k$ is strictly isotonic and strongly strictly concave on $\SMM^k$ for $k\geq 2$.
The following property is then a consequence of Lemma~\ref{L1}.

\begin{thm}
Let $\mu$ and $\zeta$ be two probability measures with respective means $a_\mu$ and $a_\zeta$ and covariances $\Sigma_\mu=\var(\mu)$ and $\Sigma_\zeta=\var(\zeta)$, $\Sigma_\mu, \Sigma_\zeta \in\SMM^k$. Then,
\bea
D_{\log\Phi_k,BR}(\mu,\zeta)=0 &\Rightarrow& a_\mu=a_\zeta \mbox{ and } \Sigma_\mu=\Sigma_\zeta \,, \\
D_{\log\Phi_k,JB}(\mu,\zeta)=0 &\Rightarrow& a_\mu=a_\zeta \mbox{ and } \Sigma_\mu=\Sigma_\zeta \,.
\eea
\end{thm}

Using \eqref{F_Phik}, we obtain that $D_{\log\Phi_k,JB}(\mu,\zeta)$ corresponds to the simplicial distance between $\mu$ and $\zeta$ introduced in \cite{PWZ-2018-jmva},
\be
D_{\log\Phi_k,JB}(\mu,\zeta) &=& \frac12\, \left\{ \tr\left[\frac{\nabla_{\Phi_k}(\Sigma_\mu)}{\Phi_k(\Sigma_\mu)}\Sigma_\zeta\right]+
\tr\left[\frac{\nabla_{\Phi_k}(\Sigma_\zeta)}{\Phi_k(\Sigma_\zeta)}\Sigma_\mu\right]  \right\} \nonumber \\
&& \hspace{-1cm} + \frac12\, \left\{ (a_\zeta-a_\mu)\TT \left[\frac{\nabla_{\Phi_k}(\Sigma_\mu)}{\Phi_k(\Sigma_\mu)}+\frac{\nabla_{\Phi_k}(\Sigma_\zeta)}{\Phi_k(\Sigma_\zeta)}\right] (a_\zeta-a_\mu) \right\} - k\,. \label{JB-phik}
\ee
When $k=d$, we have $D_{\log\Phi_d,JB}(\mu,\zeta)=d\, D_{\log\varphi_0,JB}(\mu,\zeta)$ and $D_{\log\Phi_d,BR}(\mu,\zeta)=d\, D_{\log\varphi_0,BR}(\mu,\zeta)$, see Section~\ref{S:general-OD-criteria}. In general, $D_{\log\Phi_k,JB}$ and $D_{\log\Phi_k,BR}$ with $k\neq d$ do not satisfy the invariance property \eqref{invariance}.
In \cite{PWZ-2018-jmva}, we show that the gradient matrix $\nabla_{\Phi_k}(M)$ is non-negative definite for any $M\in\SMM^\geq$ and any $k\in\{1,\ldots,d\}$, and is positive definite when $M\in\SMM^k$ . Moreover, $\nabla_{\Phi_k}(M)/\Phi_k(M)$ is the inverse of $M$ when $\rank(M)=k=d$ and is a generalized inverse of $M$ when $\rank(M)=k<d$.
If we write the characteristic polynomial of $M$ as
\bea
\det(\ml I_d -M) = c_1\, \ml^d + c_2\, \ml^{d-1} + \cdots + c_d\, \ml + c_{d+1} \,,
\eea
with $c_1=1$, then
\bea
\Phi_k(M) &=& (-1)^k\, \frac{k+1}{k!}\, c_{k+1} \,, \\
\nabla_{\Phi_k}(M) &=& (-1)^{k-1}\frac{k+1}{k!}\, (M^{k-1}+c_2\,M^{k-2}+\cdots+c_k \,I_d) \,.
\eea

\subsection{Other simplicial functionals}

By considering other powers than 2 in \eqref{phi_k0}, we can obtain simplicial functionals that depend on the full measure $\mu$ and not only on its covariance matrix $\Sigma_\mu$. In particular, we may obtain divergence measures that define semi-metrics, i.e., that satisfy
\be\label{LI}
 \mbox{ for any } \mu,\zeta\in\SM, \ D(\mu,\zeta)=0 \Leftrightarrow \mu=\zeta\,.
\ee

Consider in particular the dispersion measure
$$
\phi_{1,\delta}(\mu)=\Ex_\mu\{ \SV_1^\delta(X_0,X_{1})\} =\int\int \|x_0-x_1\|^\delta \, \mu(\dd x_0)\mu(\dd x_1) \,,
$$
see \cite{Bjorck56}, \cite{PWZ2016-SP}.
Direct calculation shows that its directional derivative $F_{\phi_{1,\delta}}(\mu,\zeta)$ is
\begin{eqnarray*}
F_{\phi_{1,\delta}}(\mu,\zeta) &=& 2 \int\int \|x_0-x_1\|^\delta \, (\xi-\mu)(\dd x_0)\mu(\dd x_1) \\
&=& 2 \left[\int\int \|x_0-x_1\|^\delta \, \xi(\dd x_0)\mu(\dd x_1)-\phi_{1,\delta}(\mu)\right]\,,
\end{eqnarray*}
where the first term on the right-hand side, $\int\int \|x_0-x_1\|^\delta \, \xi(\dd x_0)\mu(\dd x_1)$, corresponds to {\L}ukaszyk-Karmowski metric, see \cite{Lukaszyk2004}.
The corresponding Jeffreys-Bregman divergence is
$$
D_{\phi_{1,\delta},JB}(\mu,\zeta)=- \int\int \|x_0-x_1\|^\delta \, (\zeta-\mu)(\dd x_0)(\zeta-\mu)(\dd x_1) \,.
$$
It is called energy distance for $\delta=1$ and generalized energy distance \cite{SzekelyR2013} for $\delta\in(0,2]$.
The functional $\phi_{1,\delta}$ is concave for $\delta\in(0,2]$, strictly concave for $\delta\in(0,2)$, and the kernel $K(x_0,x_1)=-\|x_0-x_1\|^\delta$ is conditionally integrally strictly positive definite for $\delta\in(0,2)$; see \cite{schilling2012bernstein}, \cite{SriperumbudurGFSL2010}. Then, $D_{\phi_{1,\delta},JB}(\mu,\zeta)>0$ for two probability measures $\mu \neq \zeta$ having finite energy, i.e., such that $\int\int -\|x_0-x_1\|^\delta \, \mu(\dd x_0) \mu(\dd x_1) < +\infty$ and $\int\int -\|x_0-x_1\|^\delta \, \zeta(\dd x_0) \zeta(\dd x_1) < +\infty$.

Other conditionally integrally strictly positive definite kernels $K(\cdot,\cdot)$ yield strictly concave measures of dispersion
$\phi_K(\mu)=\int\int -K(x_0,x_1)  \, \mu(\dd x_0) \mu(\dd x_1)$ for probability measures, and the associated  Jeffreys-Bregman divergence is
$$
D_{\phi_K,JB}(\mu,\zeta)= \int\int K(x_0,x_1) \, (\zeta-\mu)(\dd x_0)(\zeta-\mu)(\dd x_1) \,,
$$
which corresponds to the (squared) maximum mean discrepancy between $\zeta$ and $\mu$, as defined in \cite{SejdinovicSGF2013}.
Uniformly bounded kernels are characteristic, i.e., are such that $D_{\phi_K,JB}$ satisfies \eqref{LI}, if and only if they are conditionally integrally strictly positive definite; see \cite{SriperumbudurGFSL2010}.
The question of whether simplicial dispersion functionals $\phi_{k,\delta}(\mu)= \Ex_\mu\{ \SV_k^\delta(X_0,\ldots,X_{k}) \}$ with $k\geq 2$ and $\delta\in(0,2)$ may define characteristic kernels remains an open issue.

\section{Application: testing the equality between means and covariances}\label{S:Application}

We illustrate the behaviour of the distances presented in previous sections by considering the situation where one wishes to test whether two distributions $\mu$ and $\zeta$ have the same mean and covariance, using empirical data.
We denote
\bea
\hat a_{\mu,n}^{(i)}=\frac1n\, \sum_{k=1}^n X_k^{(i)} \mbox{ and } \hat \Sigma_{\mu,n}^{(i)}=\frac{1}{n-1}\, \sum_{k=1}^n (X_k^{(i)}-a_{\mu,n}^{(i)})(X_k^{(i)}-a_{\mu,n}^{(i)})\TT
\eea
the sample mean and covariance matrix for a sample $\Xb_n^{(i)}=\{X_1,\ldots,X_n\}$ of $n$ independent $d$-dimensional vectors distributed with $\mu$, and similarly $\hat a_{\zeta,m}^{(i)}$ and $\hat \Sigma_{\zeta,m}^{(i)}$ for a sample $\Yb_m^{(i)}$ of $m$ independent vectors distributed with $\zeta$. We denote by $D(\hat\mu_n^{(i)},\hat\zeta_m^{(i)})$ the distance $D$ computed with the empirical values $\hat a_{\mu,n}^{(i)}$, $\hat \Sigma_{\mu,n}^{(i)}$, $\hat a_{\zeta,m}^{(i)}$ and $\hat \Sigma_{\zeta,m}^{(i)}$.

\subsection{ROC curves}\label{S:ROC curves}

Suppose we have $N$ pairs of independent samples $\Xb_n^{(i)}$ and $\Yb_m^{(i)}$, $i=1,2,\ldots,N$, respectively distributed $\SN(a_\mu,\Sigma_\mu)$ and $\SN(a_\zeta,\Sigma_\zeta)$.
Each pair $(\Xb_n^{(i)},\Yb_m^{(i)})$ yields an empirical distance $D(\hat\mu_n^{(i)},\hat\zeta_m^{(i)})$, with $D$ one of the distances considered above, and the $N$ pairs give an empirical estimate of the c.d.f.\ $\Fx_1$ of $D(\hat\mu_n,\hat\zeta_m)$.
Similarly, pairs $(\Xb_n^{(i)},\Xb_n^{(j)})$ yield an empirical estimate of the c.d.f.\ $\Fx_0$ of $D(\hat\mu_n,\hat\mu_n)$.

Denote by H0 the hypothesis that two given samples $\Xb_n$ and $\Yb_m$ have the same mean and covariance and H1 the hypothesis that they have different means and/or covariances.  A standard statistical test based on $D$ would compare the distance calculated for the empirical estimates $\hat a_{\mu,n}$, $\hat \Sigma_{\mu,n}$, $\hat a_{\zeta,m}$ and $\hat \Sigma_{\zeta,m}$ to some critical value $\tau$.
A plot of $1-\Fx_1$ against $1-\Fx_0$ gives the Receiver Operating Characteristic (ROC) curve for the test. It shows the value of the true positive rate against the false positive rate at various threshold settings $\tau$, and the power of the test as a function of the type-1 error of the decision rule. The Area Under the ROC Curve (AUC) gives a scalar figure of merit for the performance of the test considered.

\begin{example}\label{Ex:1}
We use pairs of samples with equal size $n=m=200$ in dimension $d=20$. Detection of different means is far easier than detection of slightly different covariances, and we take $a_\mu=a_\zeta=(1,\ldots,1)\TT$. The covariances are
\be\label{Sigma-ex1}
\Sigma_\mu=\left(
           \begin{array}{cc}
             A & 0\\
             0 & 10^{-3}\,I_{d-2} \\
           \end{array}
         \right)  \ \mbox{ and } \
\Sigma_\zeta=\left(
           \begin{array}{cc}
            \ma\, A & 0\\
             0 & 10^{-3}\,I_{d-2} \\
           \end{array}
         \right)\,,
\ee
with $A=\left(
          \begin{array}{cc}
            2 & -1 \\
            -1 & 2 \\
          \end{array}
        \right)$
and $I_{d-2}$ the $(d-2)$-dimensional identity matrix. The empirical estimates of the c.d.f.\ $\Fx_0$ and $\Fx_1$ are built from $N=1,000$ pairs of normal samples.

The left panel of Figure~\ref{F:ROC1} presents the ROC curve obtained when $\ma=1.4$ in \eqref{Sigma-ex1}, for Bhattacharyya distance $D_B$ \eqref{DB} (dashed line, bottom), $D_{\log\varphi_p,JB}$ \eqref{JB-phip} with $p=1/2$ (dotted line) and $D_{\log\Phi_k,JB}$ \eqref{JB-phik} with $k=3$ (solid line).
The right panel of Figure~\ref{F:ROC1} shows the AUC as $\ma$ varies between 1 and 2 for these three distances.
The curves obtained with Burbea-Rao divergence $D_{\log\Phi_3,BR}$ cannot be visually distinguished from those obtained with $D_{\log\Phi_3,JB}$.
Note the similar behaviours observed for $D_{\log\varphi_{1/2},JB}$ and $D_{\log\Phi_3,JB}$ on this example, both performing much better than $D_B$.
The curves obtained with $D_{\log\varphi_0,JB}$ or $D_{\log\Phi_d,JB}$ (not shown) are hardly distinguishable form those with $D_B$; distances $D_{\log\varphi_p,JB}$ with $p<0$ perform very poorly due to the high sensitivity to the presence of small eigenvalues in the spectra of $\Sigma_\mu$ and $\Sigma_\zeta$.
\fin
\end{example}

\begin{figure}[ht!]
\begin{center}
\includegraphics[width=.49\linewidth]{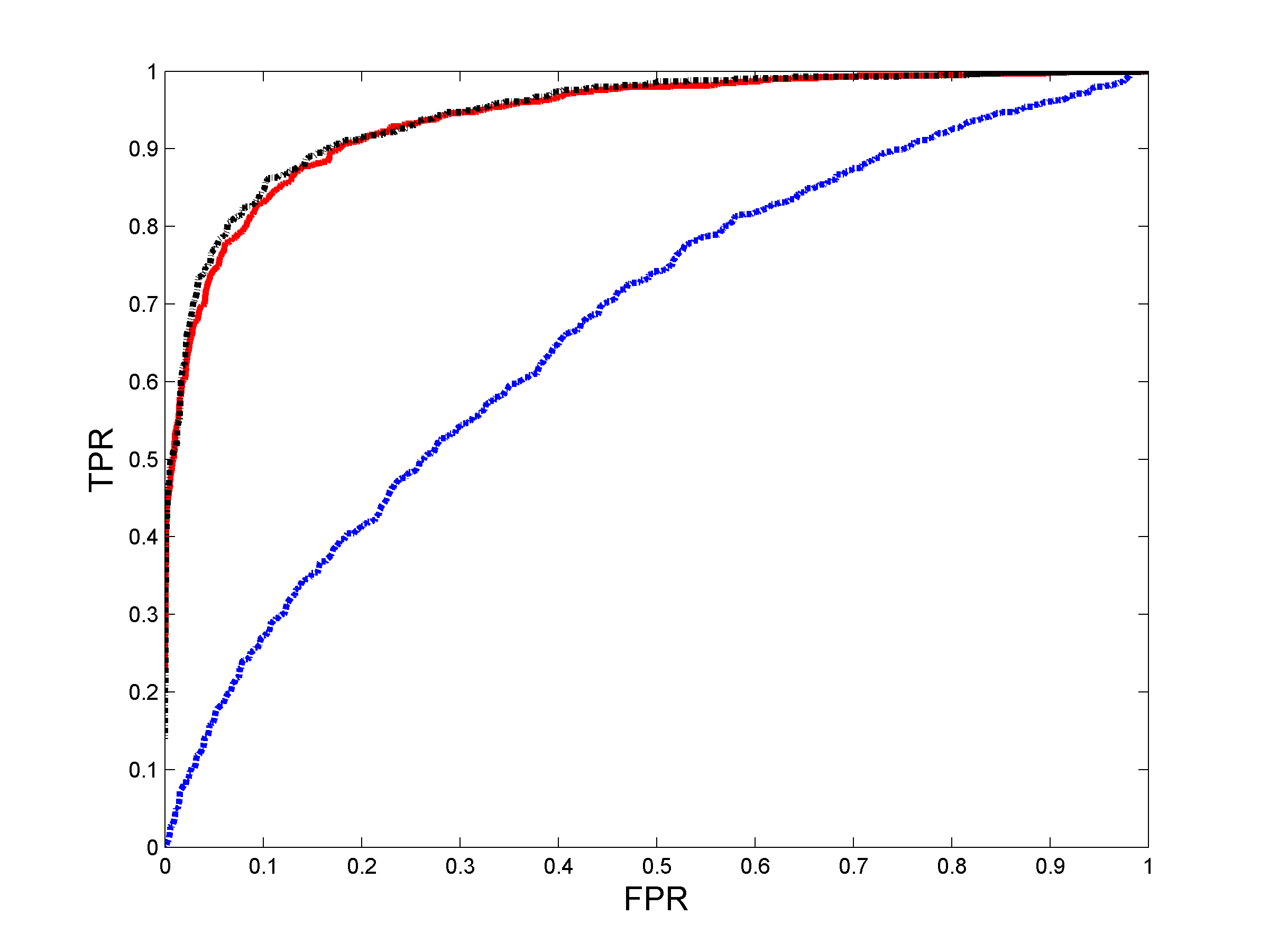} \includegraphics[width=.49\linewidth]{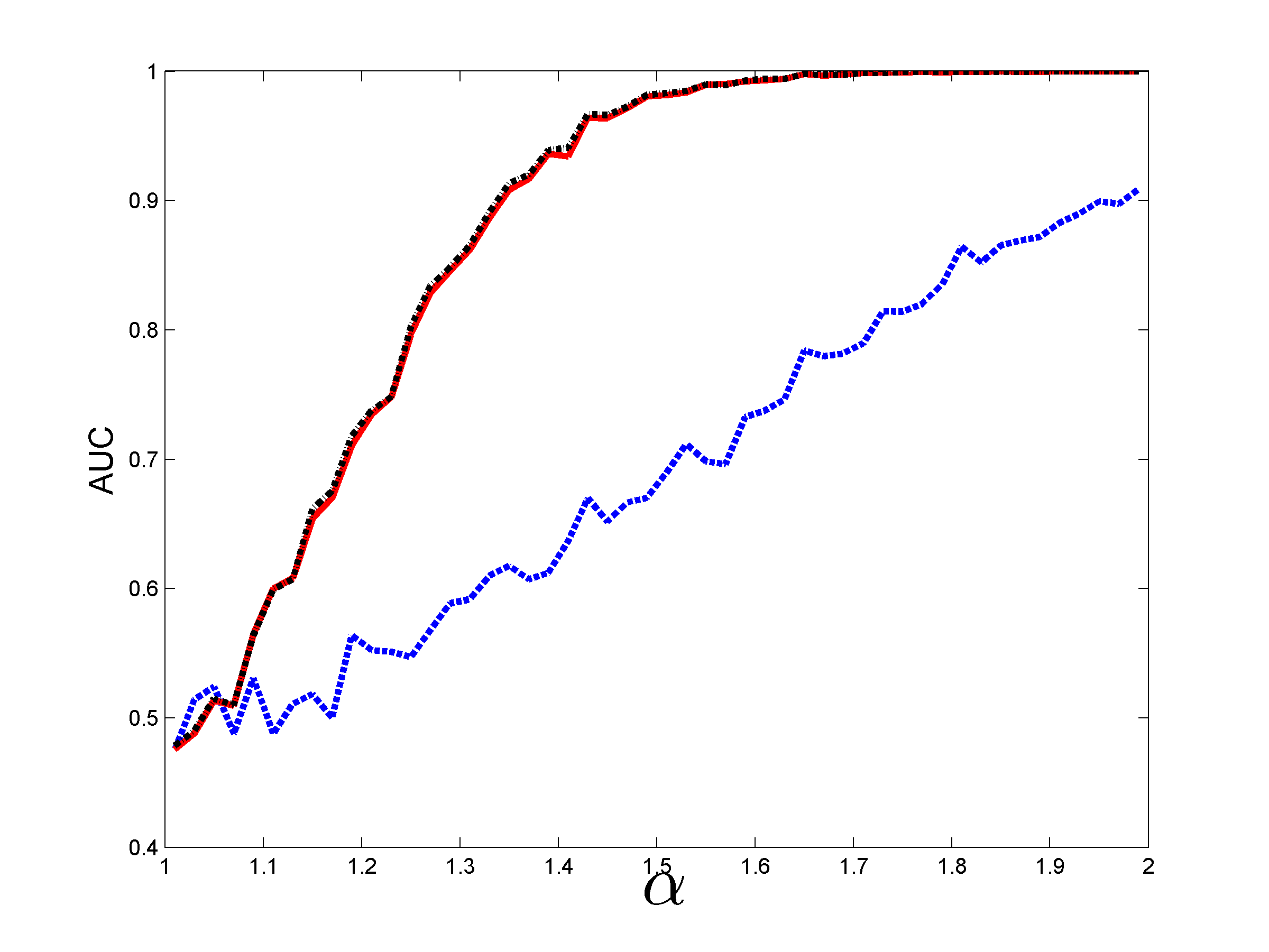}
\end{center}
\caption{\small Left: ROC curve for Bhattacharyya distance $D_B$ \eqref{DB} (dashed line, bottom), $D_{\log\varphi_{1/2},JB}$ \eqref{JB-phip} (dotted line) and $D_{\log\Phi_3,JB}$ \eqref{JB-phik} (solid line) when $\ma=1.4$ in \eqref{Sigma-ex1}. Right: Area Under the ROC Curve as a function of $\ma$.}\label{F:ROC1}
\end{figure}

As the next example illustrate, the ranking of the different methods is not always the same as in Example~\ref{Ex:1}.

\begin{example} \label{Ex:2}
We slightly modify the setting of Example~\ref{Ex:1}, and consider now covariances given by
\be\label{Sigma-ex2}
\Sigma_\mu=\left(
           \begin{array}{cc}
             A & 0\\
             0 & I_{d-2} \\
           \end{array}
         \right)  \ \mbox{ and } \
\Sigma_\zeta=\left(
           \begin{array}{cc}
            R_\theta A R_\theta\TT& 0\\
             0 & I_{d-2} \\
           \end{array}
         \right)\,,
\ee
with $A$ as in \eqref{Sigma-ex1} and $R_\theta$ the rotation matrix
$$
R_\theta = \left(
           \begin{array}{cc}
            \cos(\theta)& \sin(\theta)\\
             -\sin(\theta) & \cos(\theta) \\
           \end{array}
         \right) \,.
$$
We still have $a_\mu=a_\zeta=(1,\ldots,1)\TT$, $n=m=200$ and $d=20$.
The left panel of Figure~\ref{F:ROC2} presents the ROC curve obtained when $\theta=\pi/16$ in \eqref{Sigma-ex2}, for Bhattacharyya distance $D_B$, $D_{\log\varphi_{1/2},JB}$ and $D_{\log\Phi_3,JB}$, with the same colour code as in Figure~\ref{F:ROC1}.
The right panel of Figure~\ref{F:ROC2} shows the AUC as $\theta$ varies between 0 and $\pi/4$ for these three distances.
Again, the curves obtained with Burbea-Rao divergence $D_{\log\Phi_3,BR}$ cannot be visually distinguished from those obtained with $D_{\log\Phi_3,JB}$, and the curves obtained with $D_{\log\varphi_0,JB}$ (or $D_{\log\Phi_d,JB}$) are hardly distinguishable form those with $D_B$.
The three distances $D_B$, $D_{\log\varphi_{1/2},JB}$ and $D_{\log\Phi_3,JB}$ yield now different performances, with $D_B$ performing best, notably better than $D_{\log\Phi_3,JB}$ in particular.
\fin
\end{example}

Examples~\ref{Ex:1} and \ref{Ex:2} show the importance of being able to choose a suitable $k$ in $\{2,\ldots,d\}$ for $\log\Phi_k$, or a suitable $p$ in $[0,1)$ for $\log\varphi_p$. This is considered in the next section.

\begin{figure}[ht!]
\begin{center}
\includegraphics[width=.49\linewidth]{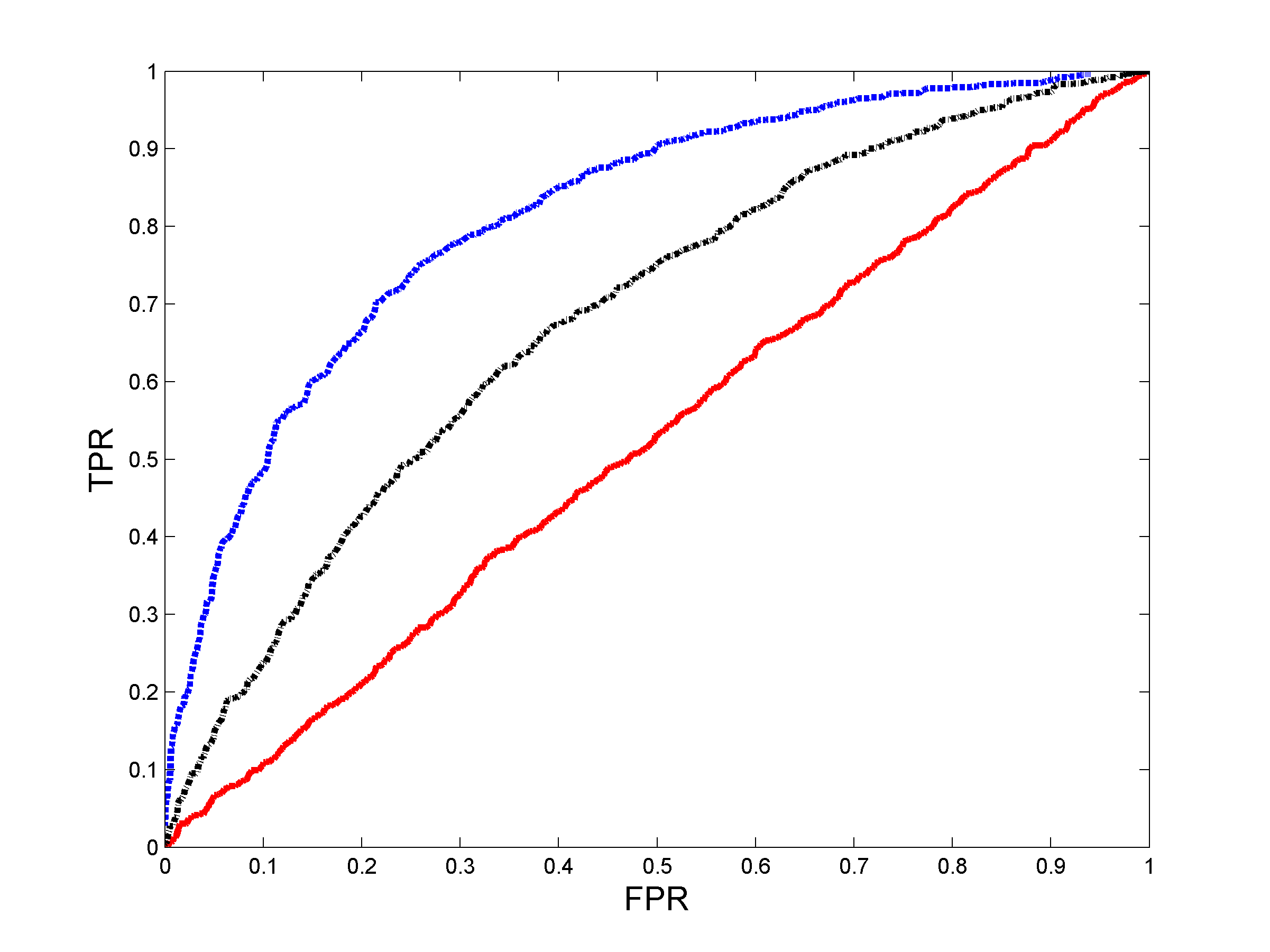} \includegraphics[width=.49\linewidth]{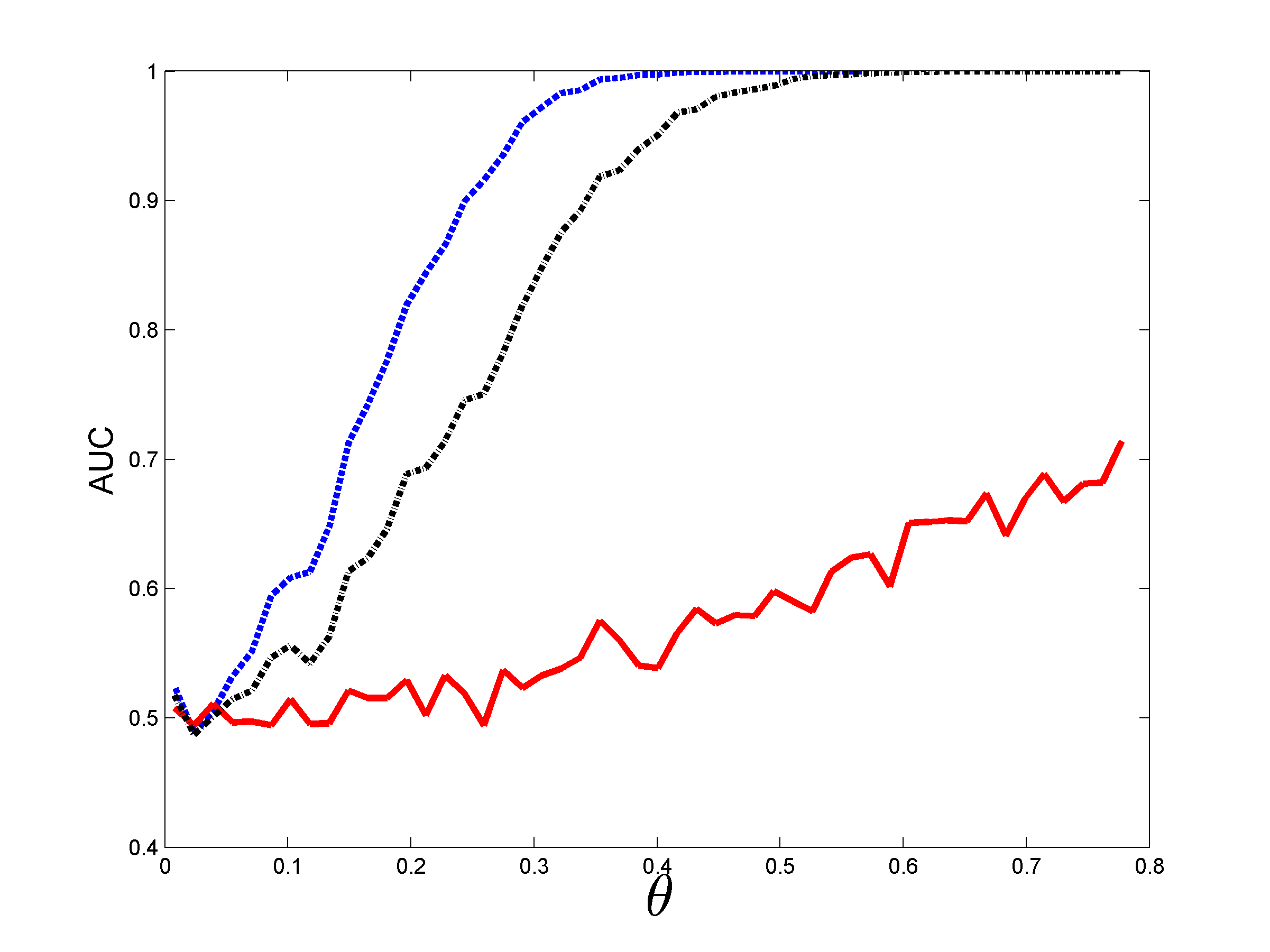}
\end{center}
\caption{\small Left: ROC curve for Bhattacharyya distance $D_B$ \eqref{DB} (dashed line, top), $D_{\log\varphi_{1/2},JB}$ \eqref{JB-phip} (dotted line, middle) and $D_{\log\Phi_3,JB}$ \eqref{JB-phik} (solid line, bottom) when $\theta=\pi/16 \approx 0.196$ in \eqref{Sigma-ex2}. Right: Area Under the ROC Curve as a function of $\theta\in[0,\pi/4]$.}\label{F:ROC2}
\end{figure}

\subsection{Choosing $k$ in $\log\Phi_k$ and $p$ in $\log\varphi_p$}\label{S:choose k}

Ideally, for testing identity between means and covariance matrices of two distributions from one sample of each, $\Xb_n$ and $\Yb_m$ say, one may use different distances and combine the test statistics obtained, $p$-values for instance. Here, we shall consider a naive approach where we first select a value $k_*$ for $k$ for distances based on $\log\Phi_k$, or $p_*$ for $p$ for distances based on $\log\varphi_p$, and then use the corresponding $k_*$, or $p_*$, in the testing procedure. A consequence of using such a simple approach is that we shall have little control of the type-I error. However, the implementation of a more precise and rigorous method would require sophisticated developments out of the scope of this paper.

%

When only one pair of samples, $\Xb_n$ and $\Yb_m$, is available, we can nevertheless generate $N$ pairs of pseudo samples from $(\Xb_n,\Yb_m)$ and use the approach of Section~\ref{S:ROC curves} to evaluate the AUC under the ROC curve for each distance considered, for several choices of $k$ and $p$. For a distance based on $\log\Phi_k$ (respectively, $\log\varphi_p$), the value $k_*$ (respectively, $p_*$) that yields the largest AUC is then selected for testing the identity between the distributions that generated the two samples $\Xb_n$ and $\Yb_m$.

For instance, we may generate pairs of pseudo samples by bootstrap. For the estimation of $\Fx_1$, each $\Xb_{n}^{(i,1)}$ (respectively, $\Yb_{m}^{(i,1)}$) is obtained by sampling with replacement within $\Xb_n$ (respectively, $\Yb_m$). For the estimation of $\Fx_0$, for each $i$ we first merge $\Xb_n$ and $\Yb_n$ into $\Zb_n=\{\Xb_n, \Yb_n\}$, then we randomly select $n$ points from $\Zb_n$, within which we sample with replacement to obtain  $\Xb_{n}^{(i,0)}$ and sample with replacement within the $m$ other points of $\Zb_n$ to obtain $\Yb_{m}^{(i,0)}$. This construction ensures that there is no intersection between $\Xb_{n}^{(i,0)}$ and $\Yb_{m}^{(i,0)}$, so that the pairs $(\Xb_{n}^{(i,0)}, \Yb_{n}^{(i,0)})$ do not look artificially too similar compared to the $(\Xb_{n}^{(i,1)}, \Yb_{n}^{(i,1)})$.

We experimentally found that sampling without replacement, as described hereafter, gives better results. Take $n'=n-r$ and $m'=m-r$, 
with $r$
sufficiently large to induce enough variability among pseudo samples. For the estimation of $\Fx_1$, each $\Xb_{n'}^{(i,1)}$ (respectively, $\Yb_{m'}^{(i,1)}$) is given by $n'$ points randomly selected within $\Xb_n$ (respectively, of $m'$ points selected within $\Yb_m$). For the estimation of $\Fx_0$, we first merge $\Xb_n$ and $\Yb_n$ into $\Zb_n=\{\Xb_n, \Yb_n\}$;  then we randomly select $n'$ points from $\Zb_n$ to form  $\Xb_{n'}^{(i,0)}$ and select $m'$ points from the remaining $n+m-n'$ points of $\Zb_n$ to form $\Yb_{m'}^{(i,0)}$. This construction ensures that there are no repetitions of points within $\Xb_{n'}^{(i,0)}$ and $\Yb_{m'}^{(i,0)}$ and no intersection between them. The value of $r$ does not need to be large: with $n=100$, $r=5$ already gives more that $75\times10^6$ different choices for $\Xb_{n'}^{(i,0)}$.


\vsp
\noindent \emph{Examples~\ref{Ex:1} and \ref{Ex:2} (continued)}
We consider again the situation of Example~\ref{Ex:1}, and draw two samples $\Xb_n$ and $\Yb_n$ from $\SN(a_\mu,\Sigma_\mu)$ and  $\SN(a_\zeta,\Sigma_\zeta)$, respectively, with $n=200$ and $\ma=1.4$ in \eqref{Sigma-ex1}. The left panel of Figure~\ref{F:select-ROC1} shows the AUC under the ROC curve for $D_{\log\Phi_k,BR}$ as a function of $k$, constructed according to the procedure above with $N=n$ and $r=5$.
The optimal $k$ is here $k_*=5$; the value of $k_*$ fluctuates depending on the random samples $\Xb_n$ and $\Yb_n$ that are drawn, with $k_*\leq 6$ in about 90\% of the cases.
The right panel of Figure~\ref{F:select-ROC1} shows the AUC under the ROC curve for $D_{\log\varphi_p,BR}$ as a function of $p$. The optimal $p_*$ varies with $\Xb_n$ and $\Yb_n$ but remains larger than 1/2 in about 90\% of the cases.
These observations suggest that in this example distances based on $\log\Phi_k$ (respectively, based on $\log\varphi_p$) perform better with small $k$ than with large $k$ (respectively, with large $p$ than with small $p$), 
which is confirmed by Figure~\ref{F:ROC1}.

\begin{figure}[ht!]
\begin{center}
\includegraphics[width=.49\linewidth]{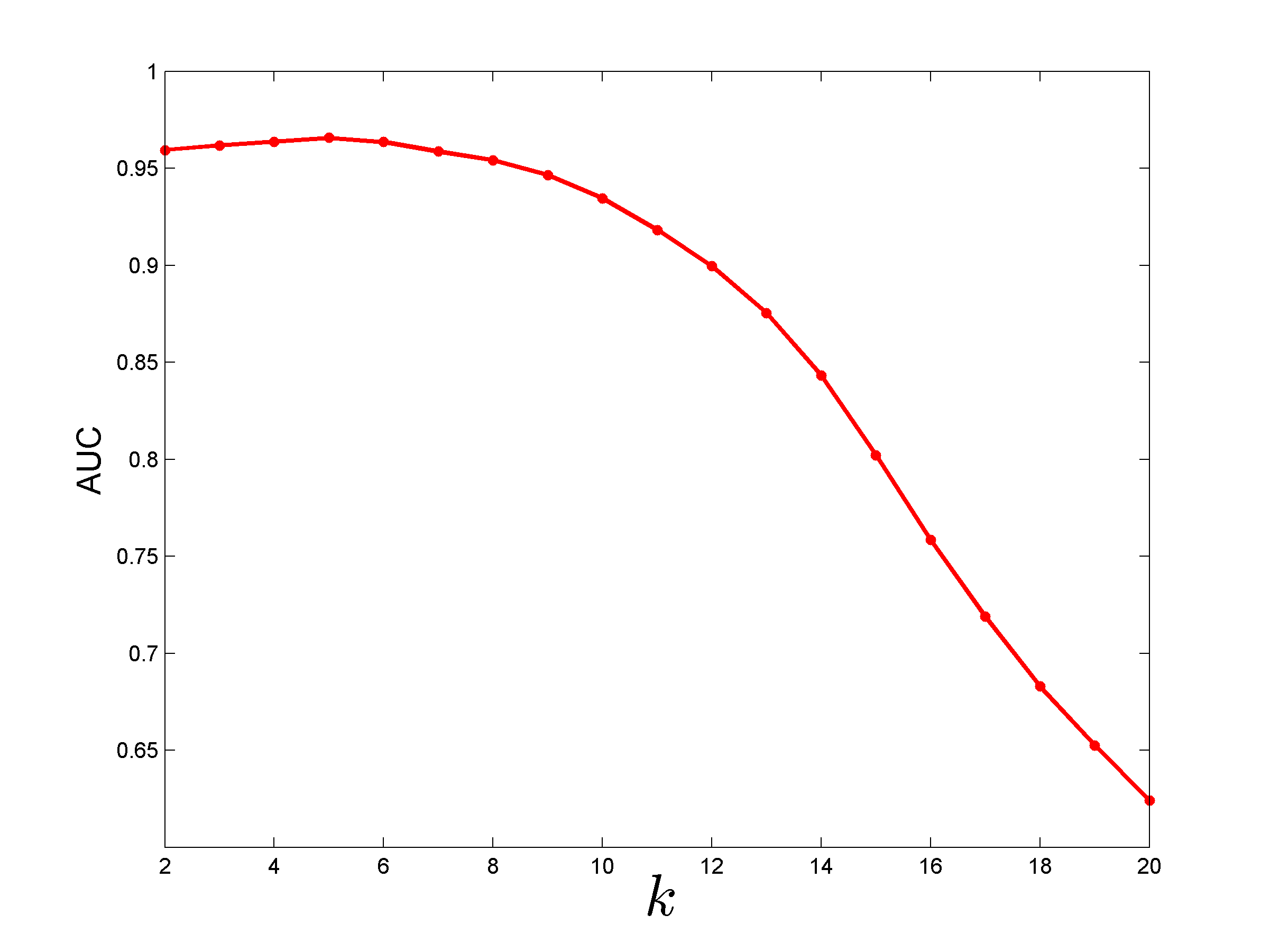} \includegraphics[width=.49\linewidth]{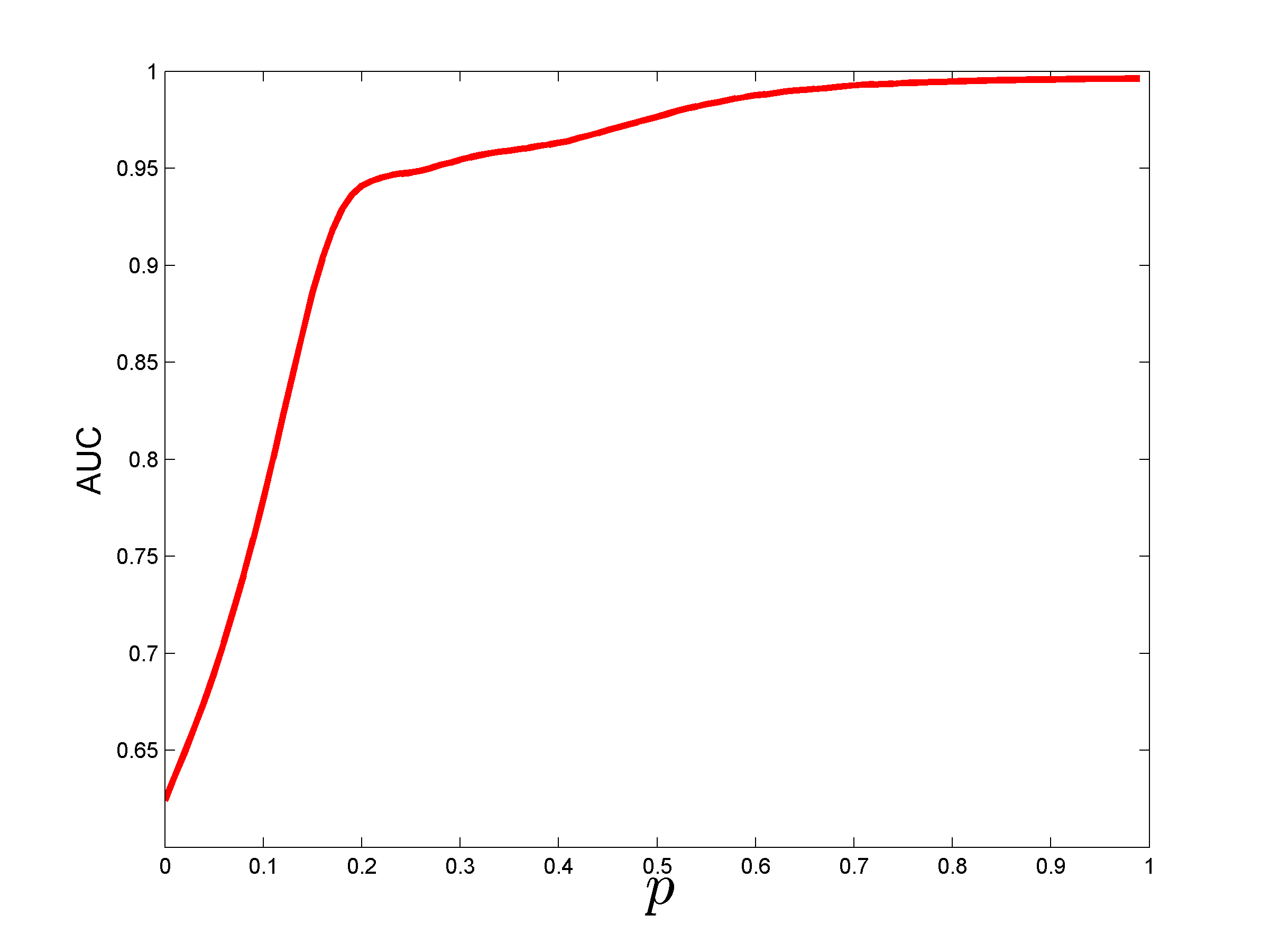}
\end{center}
\caption{\small AUC under the ROC curve for $D_{\log\Phi_k,BR}$ as a function of $k$ (Left) and for $D_{\log\varphi_p,BR}$ as a function of $p$ (Right) in the situation of Example~\ref{Ex:1} with $\ma=1.4$. We used $N=n=m=200$ pseudo samples of size $n-5 = 195$.}\label{F:select-ROC1}
\end{figure}

We repeat now the same exercice for the situation of Example~\ref{Ex:2}, with $\theta=\pi/16$. The left panel of Figure~\ref{F:select-ROC2} shows the AUC under the ROC curve for $D_{\log\Phi_k,BR}$ as a function of $k$, and the right panel shows the AUC under the ROC curve for $D_{\log\varphi_p,BR}$ as a function of $p$. We obtain $k_*=d$ and $p_*=0$, with $D_{\log\Phi_{k_*},BR}$ and $D_{\log\varphi_{p_*},BR}$ being equivalent to $D_{JS}$, see \eqref{DJS}, which coincides with Bhattacharyya distance $D_B$ when the distributions have the same mean. Figure~\ref{F:ROC2} confirms that $D_B$ is indeed a good choice in this example.
\fin

\begin{figure}[ht!]
\begin{center}
\includegraphics[width=.49\linewidth]{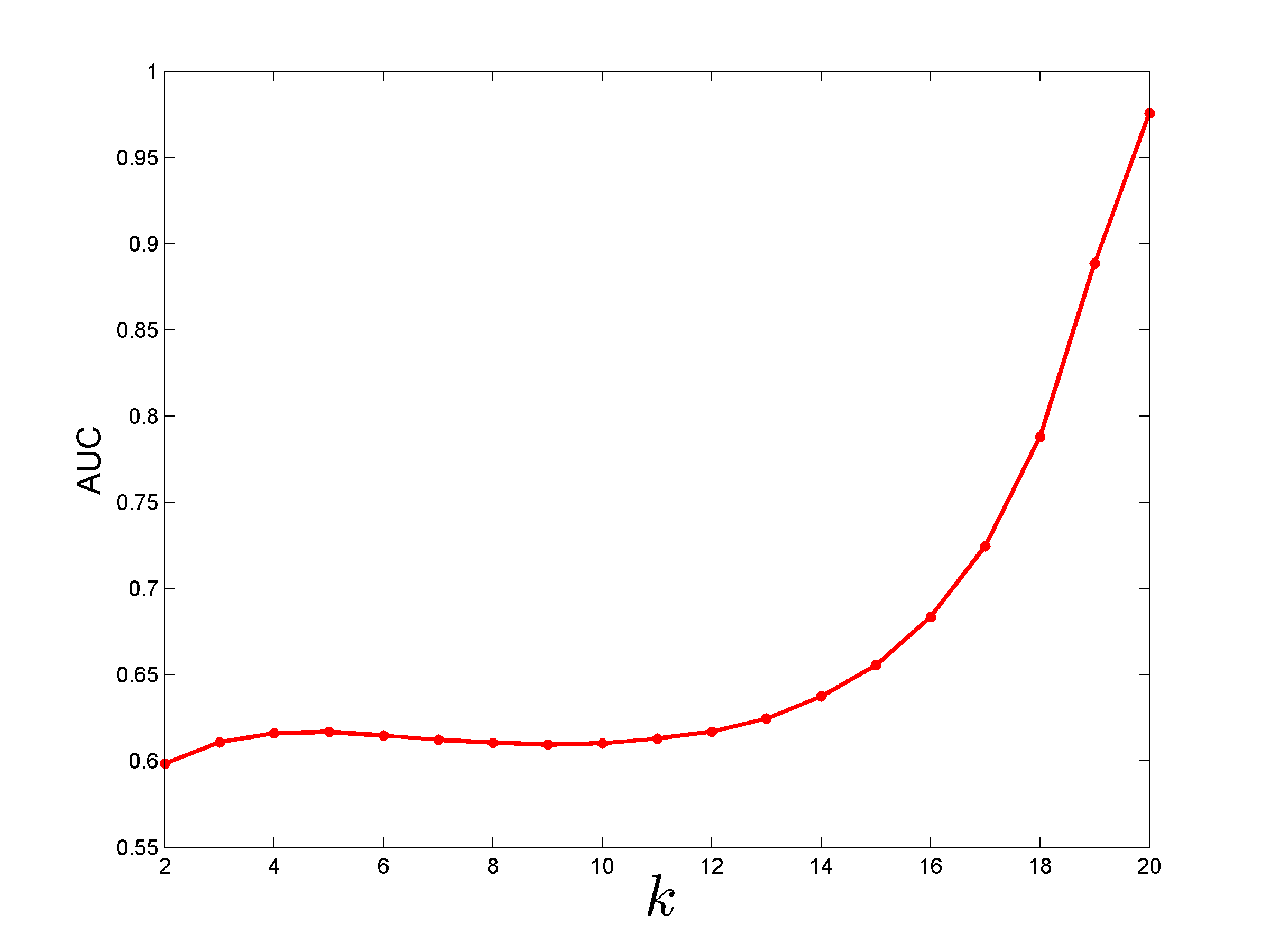} \includegraphics[width=.49\linewidth]{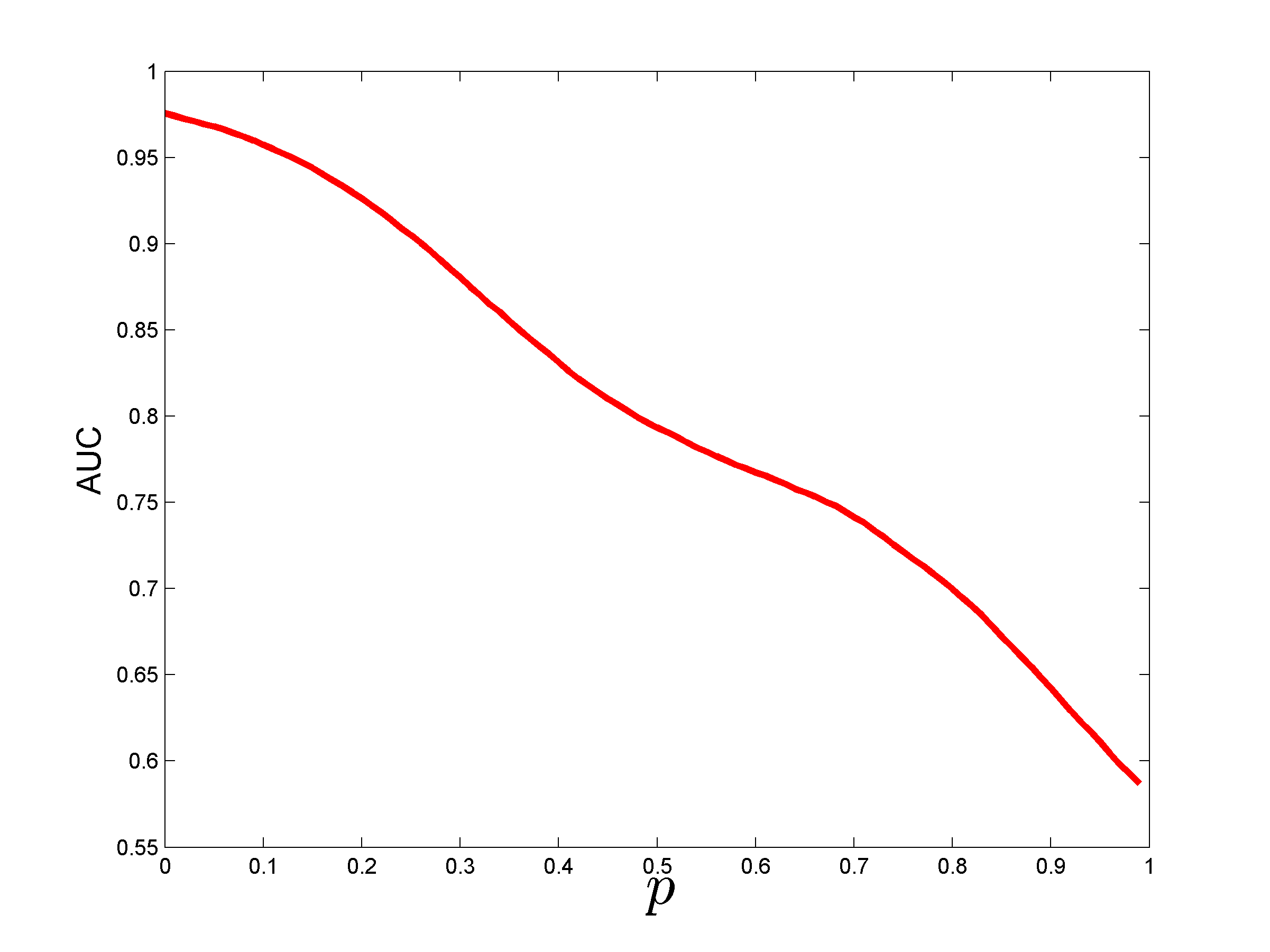}
\end{center}
\caption{\small AUC under the ROC curve for $D_{\log\Phi_k,BR}$ as a function of $k$ (Left) and for $D_{\log\varphi_p,BR}$ as a function of $p$ (Right) in the situation of Example~\ref{Ex:2} with $\theta=\pi/16$. $N=n=m=200$, $n'=m'=195$.}\label{F:select-ROC2}
\end{figure}
\subsection{Adjusting the critical value $\tau$}\label{S:adjusting tau}

We consider a simple (and incorrect) approach, where the $N$ pairs $(\Xb_{n'}^{(i,0)}, \Yb_{m'}^{(i,0)})$ of pseudo samples generated to select $k$ or $p$, see Section~\ref{S:choose k}, are also used to adjust the critical value $\tau$ of the threshold for the test statistic. Since pseudo samples have sizes $n'=n-r$ and $m'=m-r$ respectively, and distances are not invariant with respect to the sample size, we shall discard (randomly) $r$ points from $\Xb_n$ and $\Yb_m$ to compute the test statistic $D(\hat\mu_{n'},\hat\zeta_{m'})$.

Generation of bootstrap samples can be considered too. In that case, we first merge $\Xb_n$ and $\Yb_n$ into $\Zb_n=\{\Xb_n, \Yb_n\}$, then sample with replacement within $\Zb_n$, the first $n$ points give $\Xb_{n}^{(i,0)}$, the $m$ next points give $\Yb_{m}^{(i,0)}$ and we do not need to discard any data from from $\Xb_n$ and $\Yb_m$ (the test statistic is $D(\hat\mu_{n},\hat\zeta_{m})$).

\vsp
\noindent \emph{Examples~\ref{Ex:1} and \ref{Ex:2} (continued)}
Empirical results (false positive FP and true positive TP) for the situation in Example~\ref{Ex:1} are given in the left part of Table~\ref{Tb1}. For H0, $\Xb_n$ and $\Yb_n$ are normal samples generated with $\mu$; for H1, $\Xb_n$ is generated with $\mu$ and $\Yb_n$ with $\zeta$. The experiment is repeated 1000 times, the significance level is set at $5\%$.
The value of $k_*$ is searched within $\{1, \ldots,d\}$ and that of $p_*$ within $\{0,0.01, \ldots, 0.99\}$.
Results for Example~\ref{Ex:2} are indicated in the right part of the table.

In both examples, the percentage of false positive is notably larger than the targeted significance level of $5\%$, pointing out the weakness of the naive plug-in approach based on a selection of the best values $k_*$ and $p_*$ for $k$ and $p$.
Nevertheless, the percentage of true positives with a distance based on $\log\varphi_p$ or $\log\phi_k$ is much higher than for Bhattacharyya distance in Example~\ref{Ex:1} and is similar to the one with Bhattacharyya distance in Example~\ref{Ex:2}. These promising results confirm what can be observed in Figures~\ref{F:ROC1} and \ref{F:ROC2}.  \fin

\begin{table}
\begin{center}
\begin{tabular}{lrr|rr }

& \multicolumn{2}{c}{Example 1} & \multicolumn{2}{c}{Example 2} \\
                        & FP    & TP    & FP    & TP \\
\hline
$D_B$                   & 4.5   & 13.6  & 4.5   & 33.0\\
$D_{\log\varphi_p,BR}$  & 12.0  & 79.9  & 9.7   & 35.3\\
$D_{\log\Phi_k,BR}$     & 12.1  & 77.6  & 9.9   & 34.9 \\
\end{tabular}
\caption{Percentage of false positive FP (type-I error) and true positive TP obtained in 1000 repetitions for a targeted significance level of $5\%$.}
\label{Tb1}
\end{center}
\end{table}

\subsection{Example 3: comparison of means and covariances for the Wine Recognition Data}

We consider the wine data-set of the machine-learning repository, see {\url{www.mlr.cs.umass.edu/ml/datasets/Wine}}, widely used in particular as a test-bed for comparing classifiers. Here we simply consider the three classes of the data-set as three different data-sets $\Xb$, $\Yb$ and $\Zb$ and wish to test whether they significantly differ in their means and/or covariances. The data have dimension $d=14$ and the sample sizes are 59, 71 and 48. The empirical covariances have very large leading eigenvalues (larger than $10^4$) but also several eigenvalues smaller than one.

The left panel of Figure~\ref{F:winedata1} shows the value of distance $D_{\log\Phi_k,BR}$ computed for the empirical measures associated with the second and third data sets, $\Yb$ and $\Zb$ as a function of $k\in\{2,\ldots,d\}$. The curve in solid line (bottom) is when all data points are used, the one in dashed line (top) is when $r$ points are removed from each sample, see Section~\ref{S:adjusting tau}; we use $r=5$. The right panel of Figure~\ref{F:winedata1} shows (a kernel approximation of) the pdf of  $D_{\log\Phi_{10},BR}$ obtained from 1000 bootstrap samples under H0, see Section~\ref{S:adjusting tau}; the observed distance (corresponding to the value for $k=10$ on the curve in solid line on the left panel) is indicated by a vertical dashed line.
The hypothesis H0 that both samples come from distributions having the same mean and covariance is clearly rejected.
The figure obtained is similar when using sampling without replacement with $r=5$, see Section~\ref{S:choose k}. Similarly, H0 is also rejected for all other $k\in\{2,\ldots,d\}$, and when using $D_{\log\varphi_p,BR}$ for all $p=0,0.01,\ldots,0.99$. The same conclusions are obtained when comparing the distributions of $\Xb$ and $\Yb$, and $\Xb$ and $\Zb$.
They also remain unchanged when the three samples are first centered, indicating that they all have different covariances.

\begin{figure}[ht!]
\begin{center}
\includegraphics[width=.49\linewidth]{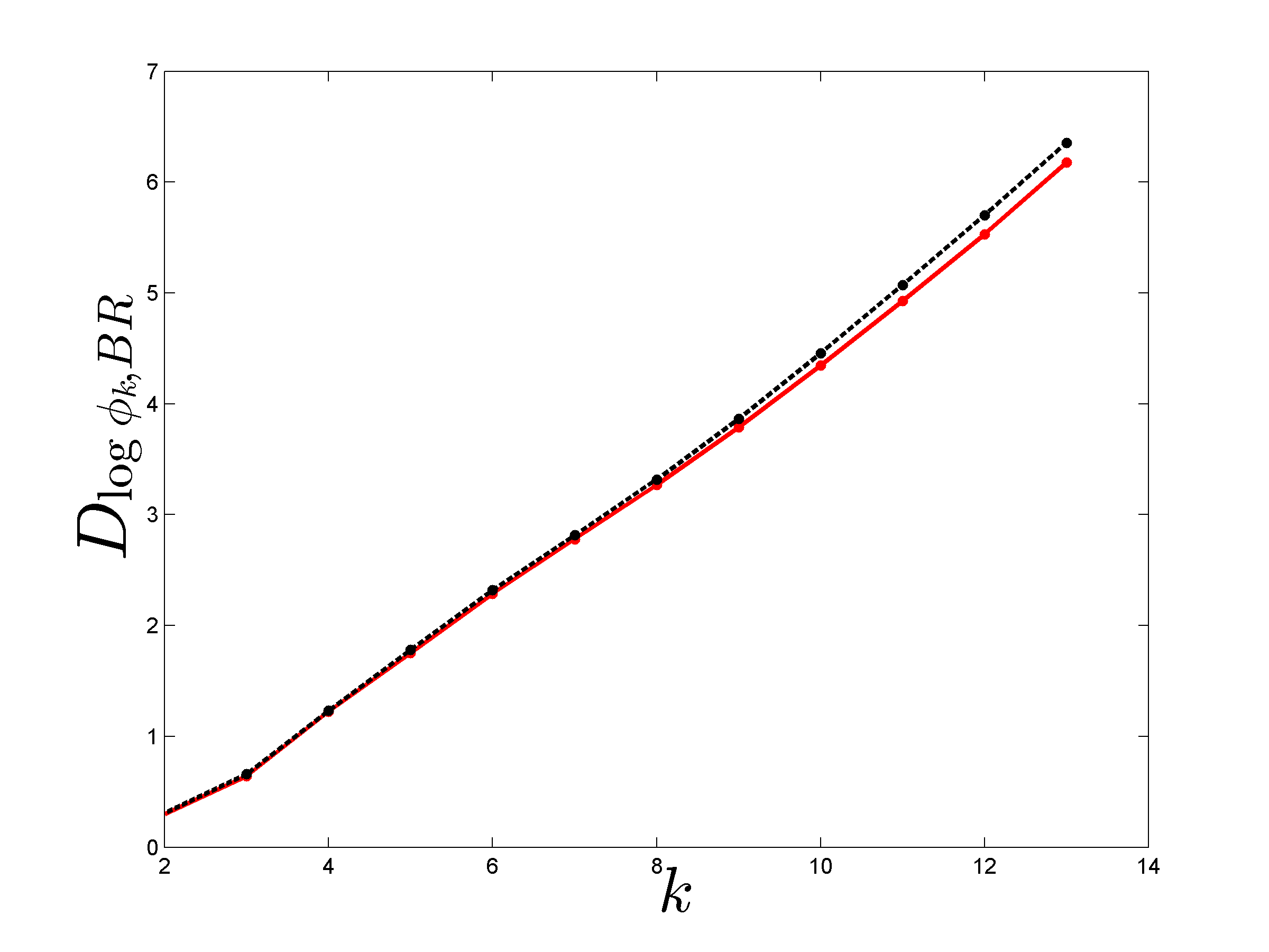} \includegraphics[width=.49\linewidth]{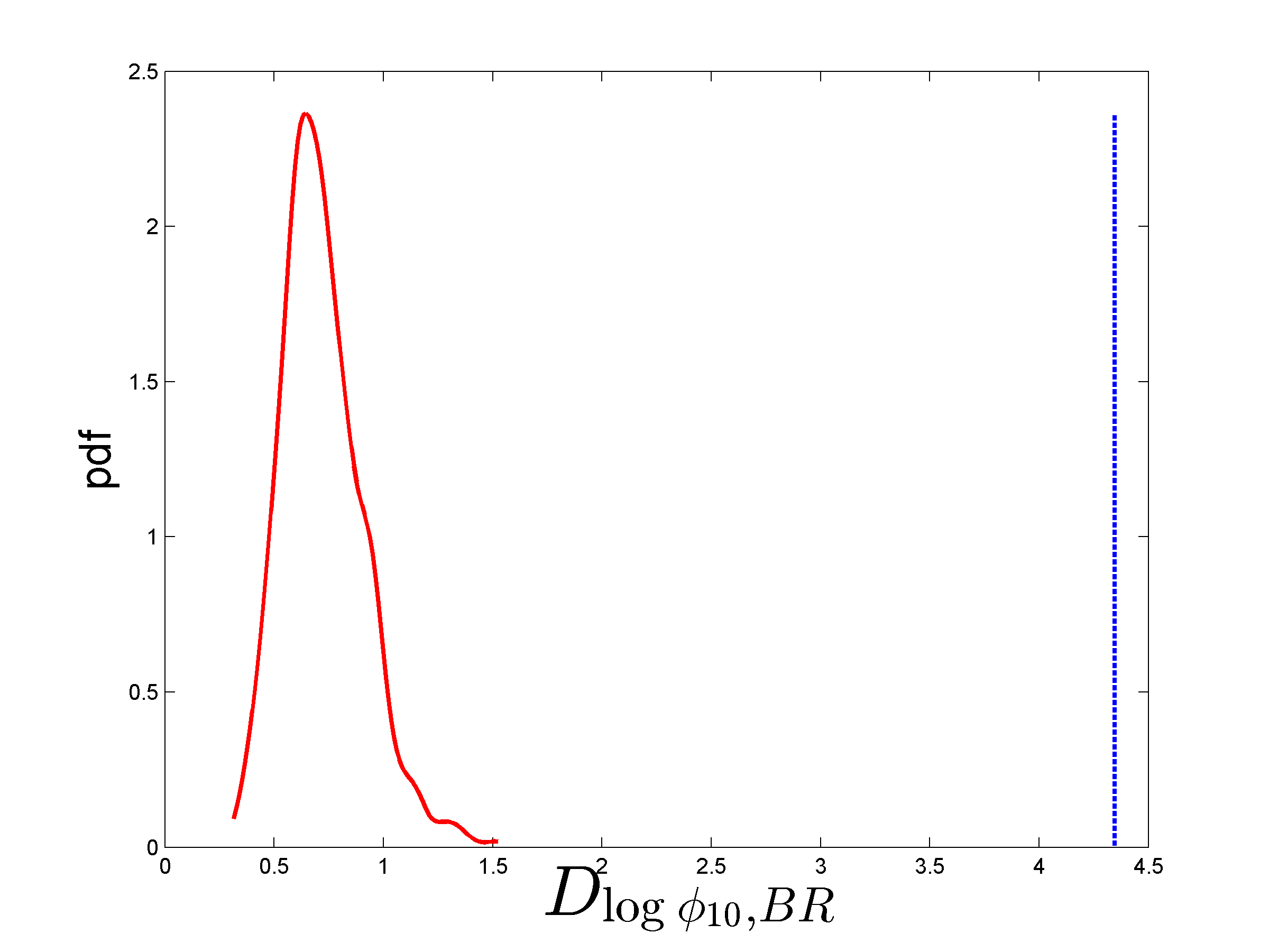}
\end{center}
\caption{\small Left: $D_{\log\Phi_k,BR}$ for the empirical measures associated with $\Yb$ and $\Zb$ as a function of $k$; solid line (bottom) when all data points are used, dashed line (top) when $n_2-r$ and $n_3-r$ points are used ($r=5$). Right: kernel approximation of the pdf of  $D_{\log\Phi_{10},BR}$ using 1000 bootstrap samples under H0 and observed value of $D_{\log\Phi_{10},BR}$ (vertical dashed line).}\label{F:winedata1}
\end{figure}

Since the data-sets $\Xb$, $\Yb$ and $\Zb$ have different sizes, in the calculation of $D_{\log \Phi_k,BR}$ we may exploit the fact that
\bea
({\widehat\Phi}_k)_n  = \frac{(n-k-1)!(n-1)^k}{(n-1)!}\, \Phi_k(\widehat \Sigma_{\mu,n})
\eea
forms an unbiased estimator of $\Phi_k(\Sigma_\mu)$ with minimum variance among all unbiased estimators, see Theorem~3.2 in \cite{PWZ2017}.
This modification does not change the conclusions above for this example.




\bibliographystyle{plain}

\bibliography{PWZ_entropy_and_distances}
\end{document}